\documentclass[11pt, a4paper, reqno]{article}

\usepackage{amssymb}
\usepackage{amsmath}
\usepackage{amsthm}
\usepackage{mathtools}
\usepackage{mathrsfs}
\usepackage[T1]{fontenc}
\usepackage[dvipsnames]{xcolor}
\usepackage{combelow}
\usepackage{enumitem}
\setenumerate{label={\rm (\alph{*})}}
\usepackage[
    colorlinks=true,
    citecolor=PineGreen,
    urlcolor=CadetBlue,
    linkcolor=RedViolet]{hyperref}
\usepackage{graphicx}
\usepackage{tikz}
\usepackage{tikz-cd}
\usepackage{pgfplots}
\usetikzlibrary{backgrounds}
\usetikzlibrary{shapes.geometric}
\pgfplotsset{compat=newest}
\usepackage[margin=1cm]{caption}
\usepackage{geometry}
\usepackage[backend=biber,style=alphabetic, maxnames=5]{biblatex}

\usepackage{algorithm}
\usepackage[noend]{algpseudocode}
\usepackage{soul}
\numberwithin{equation}{section}
\theoremstyle{definition}
\newtheorem{theorem}{Theorem}[section]
\newtheorem{lemma}[theorem]{Lemma}
\newtheorem{proposition}[theorem]{Proposition}
\newtheorem{corollary}[theorem]{Corollary}
\newtheorem{definition}[theorem]{Definition}
\newtheorem{conjecture}[theorem]{Conjecture}
\newtheorem{question}{Question}
\newtheorem{examplex}[theorem]{Example}
\newenvironment{example}
{\pushQED{\qed}\examplex}
{\popQED\endexamplex}
\theoremstyle{remark}
\newtheorem{remark}[theorem]{Remark}

\newcommand{\rc}{\mathrm{rc}}
\newcommand{\conv}{\operatorname{conv}}
\newcommand{\convD}{\conv_\mathcal{D}}
\newcommand{\convrc}{\conv_\rc}
\newcommand{\cenv}{\mathcal{C}_\rc}

\newcommand{\R}{\mathbb{R}}

\newcommand{\Mmax}{\mathsf{M}}

\newcommand{\dist}{\operatorname{dist}}

\newcommand{\Cm}{\mathcal{D}}
\renewcommand{\d}{\mathrm{d}}

\usepackage{todonotes}

\title{\vspace{-0.5cm}\textsc{\textbf{Semialgebraic rank-one convex hulls: \\ 2x2 triangular matrices and beyond}}}
\author{Chiara~Meroni and Bogdan~Rai\cb{t}\u{a}}
\date{}

\begin{document}
\maketitle

\begin{abstract}
    We prove that the rank-one convex hull of finitely many $2\times 2$ triangular matrices is a semialgebraic set, defined by linear and quadratic polynomials. We present explicit constructions for five-point configurations and offer evidence suggesting that a similar characterization does not hold in the more general setting of directional convexity.
\end{abstract}

\section{Introduction}
Quasiconvexity is  central in the study of partial differential equations (PDE), as it characterizes weak convergence in Sobolev spaces \cite{KP94}, a typical mode of convergence in both the analysis of PDE and their numerical implementations, which echoes physical measurements. However, quasiconvexity is not easily checked in general, for instance since it is known to be non-local \cite{Kr99}.
Indeed, nontrivial examples of quasiconvex functions are few and far between \cite{Sverak90:ExamplesrcFunct}. The emphasis of the present work is on the seemingly weaker notion of rank-one convexity \cite{Morrey52:QuasiConvexity}.
As the name suggests, this property of a real function on a space of matrices (resp. a set in a space of matrices) guarantees that the function (resp. the set) is convex when we restrict it to the directions given by rank-one matrices. 
\begin{definition}
A function $f:\R^{n\times m} \to \R$ is said to be \emph{rank-one convex} ($\rc$) if the restriction $t\mapsto f(A+tB)$ is a convex function for every $A\in\R^{n\times m}$ and for every rank-one matrix $B\in\R^{n\times m}$.
\end{definition}
While rank-one convexity is a weaker notion in general \cite{sverak}, under restrictive but reasonable assumptions, the two notions coincide \cite{Muller99:DiagMatrices,HKL19}. This motivates our study to better understand the structure of rank-one convex hulls.
\begin{definition}\label{def:rch}
The \emph{rank-one convex hull} of a compact set $K\subset\R^{n\times m}$ is the set
\begin{equation*}
    \convrc K = \{A\in \R^{n\times m} \colon f(A) \leq \max f(K)\, \text{for all rank-one convex } f \}.
\end{equation*}
\end{definition}
In the literature, $\convrc K$ is commonly denoted also by $K^\rc$. 
While quasiconvex hulls are not computable in general, the calculation of rank-one convex hulls is more feasible, a fact which has been exploited in the field of convex integration to construct  wild solutions of nonlinear PDE \cite{MS96,MulSve99:ConvexIntegration,KMS,Kirchheim03:RigidityGeometry,MulSve03:CounterexRegularity,KreZim06:GeomRankOneHulls,DLSz,BDLIS}. We propose studying the computation of rank-one convex hulls from algebraic and geometric perspectives, building on \cite{MatPle98:SeparateCH,Matousek01:DirectionalConvexity,FraMat09:DHullsPlane}. The following problem was stated by Bernd Sturmfels during discussions with the authors.
\begin{conjecture}\label{conj:main}
The rank-one convex hull of a finite set in $\R^{n\times m}$ is semialgebraic.
\end{conjecture}
A subset of a real vector space is \emph{semialgebraic} if it can be described as a boolean combination (namely, unions and intersections) of finitely many polynomial inequalities.
The conjecture is true in the case of \emph{separate convexity}, i.e., rank-one convexity for diagonal $n\times n$ matrices \cite{MatPle98:SeparateCH,FraMat09:DHullsPlane}, where the hulls are finite union of hyperrectangles.
In the current article, we make a first nonlinear nontrivial contribution towards a proof of Conjecture \ref{conj:main}, in the case of triangular $2\times 2$ matrices. 
\begin{theorem}\label{thm:semialg_triangular_hull}
    The rank-one convex hull of a finite set of triangular $2\times 2$ matrices is a semialgebraic set described by finitely many linear and quadratic polynomials.
\end{theorem}
Figure \ref{fig:T4triang} displays the nonlinear but semialgebraic rank-one convex hull of four triangular $2\times 2$ matrices.
\begin{figure}
    \centering
    \includegraphics[width=0.4\linewidth]{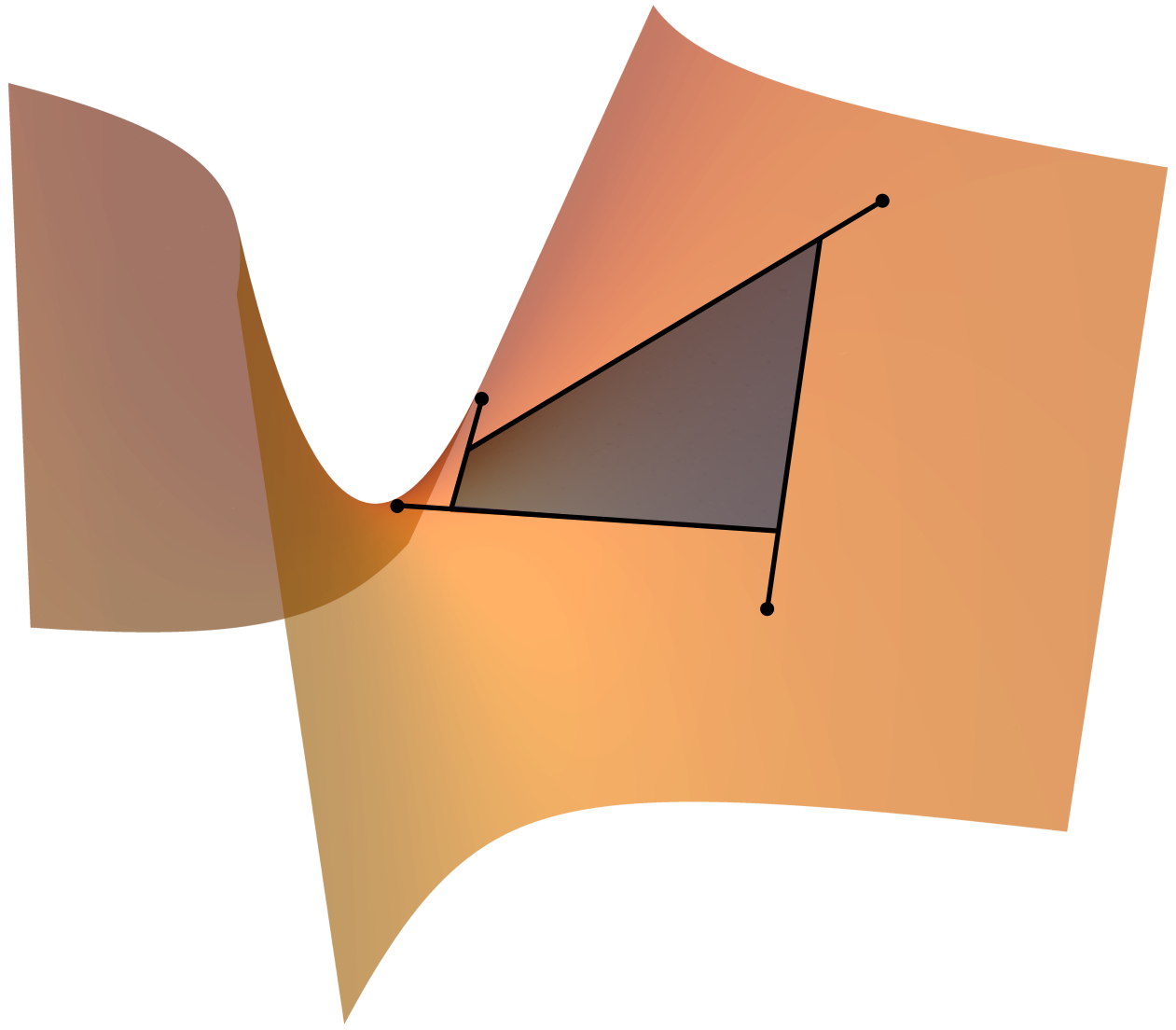}
    \caption{The rank-one convex hull (black-gray) of four triangular $2\times 2$ matrices (black dots), and the quadratic surface (orange) it lies on.}
    \label{fig:T4triang}
\end{figure}
We emphasize that, remarkably, the same result holds for quasiconvex hulls of finite sets too, as the two hulls coincide for triangular $2\times 2$ matrices \cite{HKL19}.
Theorem \ref{thm:semialg_triangular_hull} provides the first non-planar general family of nonlinear full-dimensional rank-one convex hulls. Indeed, the other nonlinear examples that are known are either $T_4$ configurations, whose hulls are lower-dimensional except in the plane (see Remark \ref{rmk:T4_all}), or the specific five-point configuration in \cite{Pompe10:5gradProblem}.

We expand on the theoretical relevance of Conjecture \ref{conj:main} and Theorem \ref{thm:semialg_triangular_hull}.
As noted above, contexts in which the quasiconvex hull is computable are quite rare and lead to the construction of interesting solutions to nonlinear PDEs. While the computation of rank-one convex hulls is considered more feasible, their characterization is still \emph{infinite}-dimensional, and it is, in fact, nontrivial to construct the rank-one convex hull of even a very simple-looking set. This is in contrast with the classical notion of convexity, where the convex hull of finitely many points is a polytope, therefore not only semialgebraic, but also describable using only linear polynomials. An affirmative answer to Conjecture \ref{conj:main} would provide a first nonlinear but \emph{finite} description of rank-one convex hulls.

\paragraph{Structure of the paper.}
Section \ref{sec:background} recalls some technical background and constructions related to rank-one convex sets, and puts this notion in the broader contest of calculus of variations in Section \ref{subsec:calculus_variations}. Then, in Section \ref{sec:plane}, we review the results for $2\times 2$ diagonal matrices, before moving to $2\times 2$ triangular matrices in Section \ref{sec:tri_matrices}.  Theorem \ref{thm:rch5pts} explicitly computes the rank-one convex hull of five triangular matrices, laying the ground for  the proof Theorem \ref{thm:semialg_triangular_hull}. Section \ref{sec:directional_convexity} is concerned with a generalization of rank-one convexity, known as directional convexity, where the rank-one convex cone is replaced by any other cone. In this context, we provide arguments that would suggest a negative answer to Conjecture \ref{conj:main} in such framework. We conclude in Section \ref{sec:open_questions} with more specific open questions, towards a solution of the conjecture.

\paragraph{Acknowledgments.} We are thankful to Bernd Sturmfels for suggesting the motivating conjecture behind this work, and for his support. We would also like to thank Andr\'e Guerra and Leonid Monin for fruitful conversations, and Vivian and Greg Kuperberg for pointing out reference \cite{KLLLT23:TilingOfRectangles}. We thank the reviewer for their detailed feedback that helped us improve the first version of this manuscript. CM thanks Georgetown University for the for generous hospitality during a research visit. CM is supported by Dr. Max R\"ossler, the Walter Haefner Foundation, and the ETH Z\"urich Foundation.

\section{Background}\label{sec:background}

Semiconvexity notions are a classic topic in the calculus of variations \cite{Mo66,Da}. The viewpoint we take in this section is inspired by \cite{Kirchheim03:RigidityGeometry}, where most proofs of the statements in this section can be found.
To start, since the cone of rank-one matrices spans the whole matrix space, a rank-one convex function is locally Lipschitz, and we can define a rank-one convexification of any other function.
\begin{definition}
Let $f:\R^{n\times m} \to \R$. Its \emph{rank-one convex envelope} is given by
\begin{equation*}
    \cenv f(A) \coloneqq \sup\{g(A)\colon g\leq f,\, g \hbox{ rank-one convex}\}.
\end{equation*}
\end{definition}
Of course, it may be possible that the envelope, as defined, equals $-\infty$ at certain points. However, either $\cenv f\equiv -\infty$, or $\cenv f>-\infty$ over the whole $\R^{n\times m}$ and $\cenv f$ is a rank-one convex function.
As defined, the rank-one convex envelope is not particularly easy to work with. We formulate a characterization in terms of a special class of discrete measures.
\begin{definition}
Given $A\in \R^{n\times m}$ an (rank-one) \emph{elementary splitting} of the Dirac measure $\delta_A$ is the measure $\lambda_1 \delta_{A_1} + \lambda_2 \delta_{A_2}$, with $\lambda_1 + \lambda_2 = 1$ and $A_2-A_1$ a rank-one matrix. We can iterate this process with an elementary splitting, for instance, of the measure $\delta_{A_1}$, and so on. We then call a (rank-one) \emph{laminate of finite order} a measure obtained by a finite number of (rank-one) elementary splittings. It is of the form
\begin{equation*}
    \nu = \sum_{i=1}^s\lambda_i \delta_{A_i}
\end{equation*}
with $\sum \lambda_i = 1$. We call its center of mass the point $\overline{\nu} = \sum \lambda_i A_i\in \R^{n\times m}$.
\end{definition}

\begin{remark}\label{rmk:crypto}
    Another related notion of hull is the \emph{lamination hull} of a set, which, in the rank-one convexity setting, is constructed by iteratively adding to the set~$K$ all line segments whose endpoints differ by a rank-one matrix. This iterative process, known as \emph{lamination}, yields an inner approximation of the rank-one convex hull of the set. Interestingly, a similar concept has recently been explored in cryptography, in the context of determining the existence of (and, if possible, constructing) secure protocols for two-party communication~\cite{BKMN22:SecureTwoParty}. In that setting, the lamination hull of a specific set provides bounds on the round complexity of the protocol. The authors further analyze the \emph{tameness} of this construction~\cite[Theorem~3]{BKMN22:SecureTwoParty}.
\end{remark}

The closure of the class of laminates of finite order is relevant to us and it can be characterized elegantly by duality with rank-one convex functions.
\begin{definition}
Let $\nu$ be a probability measure on $\R^{n\times m}$. We say that $\nu$ is a (rank-one) \emph{laminate} if the Jensen-type inequality
$$
\langle \nu , f \rangle\geq f(\bar \nu)
$$
holds for all rank-one convex functions $f$. Here $\bar\nu=\int x\, \d\nu(x)$ denotes the {barycenter} of $\nu$.
\end{definition}
Then, we can approximate any laminate $\nu$ by finite order laminates $\{\nu_i\}_i$ in the weak-$^*$ topology, namely 
$$
\langle \nu_j , f \rangle\rightarrow \langle \nu , f \rangle\quad\text{for all }f\in C_0(\R^{n\times m}),
$$
where $C_0(\R^{n\times m})$ denotes the space of continuous functions that vanish at infinity, endowed with the supremum norm (i.e., the uniform closure of the space of continuous functions with compact support).
We can characterize the rank-one convex envelope of a function in terms of laminates.
\begin{lemma}{\cite[Lemma 2.2]{MulSve03:CounterexRegularity}}
Let $f\colon\R^{n\times m}\rightarrow\R$ be a continuous function. Then 
\begin{equation*}
    \cenv f(A) = \inf \{\langle \nu , f \rangle\colon\nu \hbox{ laminate of finite order, } \overline{\nu} = A\}.
\end{equation*}
\end{lemma}

By Definition \ref{def:rch} of the rank-one convex hull of a set $K$, it is clear that $\convrc K \subset \conv K$ and it is also immediate that we can reduce to quantify over rank-one convex functions that are zero on $K$, namely
\begin{equation}\label{eq:rchull}
    \convrc K=\{A\in\R^{n\times m}\colon f(A)\leq 0\,\text{ for all } f\hbox{ rank-one convex with }f|_K\equiv0\}.
\end{equation}
Actually, we can characterize this set in another useful way, by using the vanishing locus of a single (but in general complicated) rank-one convex function.
Let $d_K\colon\R^{n\times m}\to \R$ denote the distance function from $K$, let $1\leq p<\infty$ and $K\subset\R^{n\times m}$ be compact. Then
\begin{equation*}
    \convrc K = \{A\in \R^{n\times m}\colon\cenv d_K^p (A) = 0\}.
\end{equation*}
Therefore, the rank-one convex hull of $K$ is the zero-set of
\begin{equation}\label{eq:inf_laminate}
    \cenv d_K^2(A) = \inf \{\langle \nu , d_K^2 \rangle\colon\nu \hbox{ laminate of finite order, } \overline{\nu} = A\}.
\end{equation}
The \emph{squared} distance function $d_K^2$ has the appeal that, in the case when $K=\{K_1,\ldots, K_N\}$, it can be expressed as the minimum of finitely many simple polynomials:
$$
d_K^2(A)=\min_{i=1,\ldots,N}\|A-K_i\|^2.
$$
Since $K$ is compact, for every point in $\convrc K$ there exists a sequence $(\nu_j)_j$ of laminates of finite order that are weakly-* convergent to a laminate $\nu$ with barycenter the given point and support contained in $K$.

We conclude the section by recalling the notion of $T_N$ configuration, (see e.g., \cite[Definition 1]{Szekelyhidi05:rch2x2}) which will be one of the building blocks in the rest of the paper.
\begin{definition}
We say that the points $K_1,\ldots ,K_N \in \R^{n\times m}$ form a \textbf{$T_N$ configuration} if there exist $P\in \R^{n\times m}$, $C_1,\ldots ,C_M$ of rank-one, $\alpha_1,\ldots ,\alpha_N \in \R$ such that
\begin{align*}
    K_1 &= P + \alpha_1 C_1 \\
    K_2 &= P + C_1 + \alpha_2 C_2 \\
    &\vdots \\
    K_N &= P + C_1 + \ldots + C_{N-1} + \alpha_N C_N
\end{align*}
with $\alpha_i \geq 1$ for every $i$ and $\sum_{i = 1}^N C_i = 0$.
\end{definition}
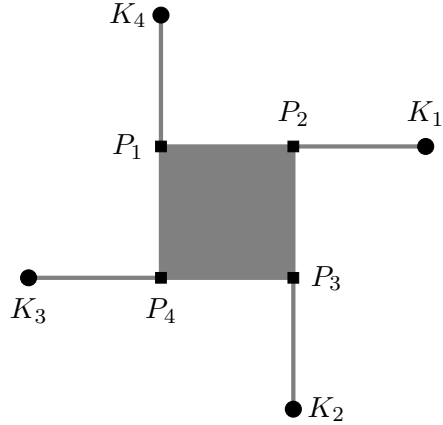
\begin{figure}[ht]
    \centering
    \begin{tikzpicture}[outer sep=0]
\begin{axis}[
width=2.4in,
height=2.4in,
scale only axis,
xmin=-3.5,
xmax=3.5,
ymin=-3.5,
ymax=3.5,
ticks = none, 
ticks = none,
axis background/.style={fill=white},
axis line style={draw=none} 
]

\addplot [color=black!50!white,ultra thick,solid,forget plot]
  table[row sep=crcr]{%
1   -1\\
1   -3\\
};

\addplot [color=black!50!white,ultra thick,solid,forget plot]
  table[row sep=crcr]{%
1   1\\
3   1\\
};

\addplot [color=black!50!white,ultra thick,solid,forget plot]
  table[row sep=crcr]{%
-1   1\\
-1   3\\
};

\addplot [color=black!50!white,ultra thick,solid,forget plot]
  table[row sep=crcr]{%
-1   -1\\
-3   -1\\
};

\addplot [color=black!50!white,fill=black!50!white,ultra thick,solid,forget plot]
  table[row sep=crcr]{%
1   -1\\
1   1\\
-1   1\\
-1   -1\\
1   -1\\
};

\addplot[only marks,mark=*,mark size=3pt,black,
]  coordinates {
    (3,1) (-1,3) (-3,-1) (1,-3)
};

\addplot[only marks,mark=square*,mark size=2pt,black,
]  coordinates {
    (1,-1) (1,1) (-1,1) (-1,-1)
};

\node at (-1.5,3) {$K_4$}; 
\node at (-3,-1.5) {$K_3$}; 
\node at (1.5,-3) {$K_2$}; 
\node at (3,1.5) {$K_1$};

\node at (-1.5,1) {$P_1$}; 
\node at (-1,-1.5) {$P_4$}; 
\node at (1.5,-1) {$P_3$}; 
\node at (1,1.5) {$P_2$}; 

\end{axis}
\end{tikzpicture}
    \caption{The classical Tartar square, a planar $T_4$ configuration.}
    \label{fig:T4}
\end{figure}
For $N\geq 4$, the rank-one convex hull of $T_N$ configurations in $\R^{n\times m}$ is nontrivial (see, e.g., \cite[Lem.~2]{Szekelyhidi05:rch2x2}).
\begin{remark}\label{rmk:degenerateT4}
    Sometimes in the literature, $T_4$ configurations are defined with the requirement $\alpha_i > 1$. We relax this condition to allow equality, in order to include \emph{degenerate} $T_4$ configurations, where some of the points in $K$ are rank-one connected.
\end{remark}

\begin{example}\label{ex:T4}
    As an example of a $T_4$ configuration, consider the by\textcolor{blue}{-}now classical \emph{Tartar square} \cite{Tartar93:SeparatelyConvex} in Figure \ref{fig:T4} for
    \[
    K=\{(3,1),\,(1,-3),\,(-3,-1),\,(-1,3)\},
    \]
    where we identify $2\times2$ diagonal matrices with vectors in $\R^2$. We have $P = (-1,1)$, $C_1 = (2,0)$, $C_2 = (0,-2)$, $C_3 = (-2,0)$, $C_4 = (0,2)$, $\alpha_i = 2$ for $i=1,\ldots,4$ and, for convenience, we denote $P_i = P + C_1 + \ldots + C_{i-1}$, with $P_1=P$.
    The rank-one convex hull $\convrc K$ is given by the union of the square $[-1,1]^2$ and the rank-one segments connecting the vertices of the square to the points of $K$.

    To prove that, e.g, $P_1$ belongs to $\convrc K$ we can construct a sequence of laminates by splitting the measure centered at $P_1$ along the vertical and horizontal segments bounding the $T_4$, as follows:
    \begin{gather*}
        \nu_1 = \delta_{P_1}, \quad \nu_2 = \frac{1}{2} \delta_{P_4} + \frac{1}{2} \delta_{K_4}, \quad \nu_3 = \frac{1}{4} \delta_{P_3} + \frac{1}{4} \delta_{K_3} + \frac{1}{2} \delta_{K_4}, \quad \ldots \\
        \nu_k = \tfrac{1}{2^{k-1}}\delta_{P_{i_k}} + \sum_{j=0}^{\frac{k-2}{4}} \tfrac{1}{2^{n-4j-1}}\delta_{K_{i_k}} + \sum_{j=0}^{\frac{k-3}{4}} \tfrac{1}{2^{n-4j-2}}\delta_{K_{i_k+1}} + \sum_{j=0}^{\frac{k-4}{4}} \tfrac{1}{2^{n-4j-3}}\delta_{K_{i_k+2}} + \sum_{j=0}^{\frac{k-5}{4}} \tfrac{1}{2^{n-4j-4}}\delta_{K_{i_k+3}},
    \end{gather*}
    where 
    \[
    i_k = 
    \begin{cases}
        1 & k\equiv 1 \mod 4, \\
        2 & k\equiv 0 \mod 4, \\
        3 & k\equiv 3 \mod 4, \\
        4 & k\equiv 2 \mod 4. \\
    \end{cases}\qedhere
    \]
    It can be seen in this example that sending $k\to \infty$ produces the weakly-* limit $\nu_\infty$ with support on $K$ only (the weight in front of $P_{i_k}$ goes to zero in the limit).
\end{example}
Example \ref{ex:T4} is a planar situation, which makes it a quite special case among the class of all $T_4$ configurations in $\R^{2\times 2}$. These are classified in \cite[Thm.~2]{Szekelyhidi05:rch2x2}, see also \cite{Kirchheim03:RigidityGeometry}.
It may be interesting to remark that, in contrast with the fact that the only inclusion-minimal configurations for rank-one convexity in $\R^{2\times 2}$ are $T_4$ configurations, it is shown in \cite{Matousek01:DirectionalConvexity} that for separate convexity in $\R^3$, there are inclusion-minimal $T_N$ configurations of arbitrary size, leading to an infinite Carath\'{e}odory number. 
In our case of triangular matrices, the Carath\'{e}odory number is also infinite (see \cite[Example 3.5 $(ii)$]{MatPle98:SeparateCH}, and the discussion after Theorem \ref{thm:rch5pts}), which highlights the nontriviality of the question at hand and shows how much these concepts differ from usual convexity.

\subsection{Rank-one convexity from calculus of variations}\label{subsec:calculus_variations}
Consider the following partial differential inclusion:
\begin{equation}\label{eq:PDE_curl}
    \begin{cases}
        \mathcal{A} v (x) = 0 & \hbox{for } x\in \R^m, \\
        v(x) \in K & \hbox{for almost every } x \in \R^m, 
    \end{cases}
\end{equation}
where $\mathcal A$ is a \textit{linear} partial differential operator and $K$ is a set encoding the nonlinear geometry of the problem. This framework, introduced by Tartar in the 1970s \cite{Ta79,Ta83,Ta05}, unifies several models in continuum mechanics, including the aforementioned developments in \cite{MS96,MulSve99:ConvexIntegration,DLSz,BDLIS}.
Our contribution is relevant to the case where  $\mathcal{A}$ is the curl operator $\mathcal{A} v = \partial_k v_{ij} - \partial_j v_{ik}$, $v = (v_{ij})_{i,j} : \R^m \to \R^{n\times m}$, $K$ finite. Then, $\mathcal{A} v = 0$ can be solved for $v = \nabla f$, for some $f:\R^m\to\R^n$. Therefore, asking if there exists a nonconstant Lipschitz solution to \eqref{eq:PDE_curl} is equivalent to ask whether there exists such a function $f$ satisfying $\nabla f \in K$. 
The construction of solutions of the above system of PDE’s involves computing the $\mathcal{D}$-convex hull of $K$, where $\mathcal{D}$ is a cone associated to the operator $\mathcal{A}$, known as the \emph{wave cone}. In the case of the curl, the wave cone is the cone of rank-one $(n \times m)$ matrices, and one talks about rank-one convexity. This motivates the study of the rank-one convex hull of finitely many points. 
In the context of more general operators and wave cones, a related problem was studied recently in \cite{SoTi}, to investigate the so-called four-state problem and the rigidity of linear partial differential operators. 
In this broader context, the wave cone of $\mathcal A$, leads to an appropriate notion of directional convexity, which we will treat in Section \ref{sec:directional_convexity}.

\section{The planar grid explained}\label{sec:plane}

In this section we recall some results for rank-one convexity for $2\times 2$ diagonal matrices, also known as separate convexity in $\R^2$. In this case, the cone of rank-one matrices is identified with $\mathcal{D}_{\rm diag} = \{xy=0\} \subset \R^2$.
We introduce some notation. Let $K = \{K_1,\ldots,K_N\}\subset \R^2$ be a finite set of points. Then, the \emph{grid} associated to $K$ is the line arrangement consisting of the $2N$ lines
\[
\bigcup_{i=1}^N \{ K_i + (0,1)\cdot\R,\: K_i + (1,0)\cdot\R \} = \bigcup_{i=1}^N (K_i + \mathcal{D}_{\rm diag}).
\]
The \emph{faces} of the grid (hence vertices, edges, $\ldots$) associated to $K$ are the cells of the line arrangement. By abuse of notation, we use the term \emph{grid of $\convrc K$} for the grid associated to $K$ intersected with $\convrc K$.
The following is a rephrasing of \cite[Proposition 5.1]{MatPle98:SeparateCH} and a special case of \cite[Theorem 1.1]{FraMat09:DHullsPlane}.
\begin{theorem}\label{thm:plane}
    The rank-one convex hull of a finite set $K$ of $2\times 2$ diagonal matrices is a finite union of closed faces of the grid associated to $K$.
\end{theorem}
Finiteness follows for instance from the fact that the rank-one convex hull is contained in the convex hull. We refer to either \cite[Proposition 5.1]{MatPle98:SeparateCH} or \cite[Theorem 1.1]{FraMat09:DHullsPlane} for the proof. Instead, we emphasize that such a result is constructive. Given a finite set $K\subset \R^2$, draw the cones $\mathcal{D}_{\rm diag}$ centered at all the points in $K$ to form the grid associated to $K$. Then, simply remove all the vertices of the grid that have at most one neighbor in each coordinate direction. Iterating this process until no more vertices can be removed returns the rank-one convex hull of $K$. We make this precise in Algorithm \ref{alg:plane_hull}.
A refined implementation in \cite[Proposition 5.4]{MatPle98:SeparateCH} shows that $\convrc K$ can be computed in $O(N\log N)$ time. 
\begin{algorithm}[!ht]
    \caption{Rank-one convex hull of finitely many diagonal $2\times 2$ matrices.}
    \label{alg:plane_hull}
    \textsc{Input:} $K=\{K_1,\ldots,K_N\} = \{(x_1,y_1),\ldots,(x_N,y_N)\}\subset \R^2$, $N\in \mathbb{N}$.\\
    \textsc{Output:} $\convrc K$.
    \begin{algorithmic}[1]
    \State $x_{\min}, \; x_{\max} \gets \min x_i, \; \max x_i$
    \State $y_{\min}, \; y_{\max} \gets \min y_i, \; \max y_i$
    \State $B_0 \gets [ x_{\min} , x_{\max} ] \times [ y_{\min} , y_{\max} ]$
    \State $V \gets $ all the vertices of the grid associated to $K$
    \For{$i > 0$}
        \For{$v\in V \setminus K$} \label{alg_step:conv_vert}
            \If {both grid lines at $v$ have points of $V$ only on one side of $v$}
                \State remove $v$ from $V$
            \EndIf
        \EndFor
        \If{no $v$ has been removed}
            \State stop all the loops
        \EndIf
        \State $B_i \gets $ union of the closed faces of the grid with extreme points in $V$
        \State $i \gets i+1$
    \EndFor
    \State \Return $B_i$
    \end{algorithmic}
\end{algorithm}

A consequence of Theorem \ref{thm:plane} is that the rank-one convex hull of a finite set of $2\times 2$ diagonal matrices is a semialgebraic set, whose boundary can be described by finitely many line segments and isolated points.

\begin{example}\label{ex:plane_grid}
Let 
\begin{equation*}
        K = \left\lbrace
        \begin{pmatrix}
            3 & 0 \\
            0 & 1
        \end{pmatrix},
        \begin{pmatrix}
            1 & 0 \\
            0 & -3
        \end{pmatrix},
        \begin{pmatrix}
            -3 & 0 \\
            0 & -1
        \end{pmatrix},
        \begin{pmatrix}
            -1 & 0 \\
            0 & 3
        \end{pmatrix},
        \begin{pmatrix}
            2 &0 \\
            0 & 2
        \end{pmatrix}
        \right\rbrace \subset \R^2.
    \end{equation*}
    Its rank-one convex hull can be computed using Algorithm \ref{alg:plane_hull}, with only three iterations. The initial candidate is the box $[-3,3]\times[-3,3]$, and Figure \ref{fig:grid} displays all the steps in the algorithm. The right-most shape is the rank-one convex hull of $K$, and it can be seen from the figure that no other grid points are allowed to be removed.
    \begin{figure}[ht]
        \centering
        \begin{tikzpicture}[scale=0.5, outer sep=0]

    \fill[black!15!white] (3, -3) -- (3, 3) -- (-3, 3) -- (-3, -3) -- cycle;

    \draw[black!50!white, ultra thick] (3, -3) -- (3, 3);
    \draw[black!50!white, ultra thick] (2, -3) -- (2, 3);
    \draw[black!50!white, ultra thick] (1, -3) -- (1, 3);
    \draw[black!50!white, ultra thick] (-1, -3) -- (-1, 3);
    \draw[black!50!white, ultra thick] (-3, -3) -- (-3, 3);
    \draw[black!50!white, ultra thick] (-3, 3) -- (3, 3);
    \draw[black!50!white, ultra thick] (-3, 2) -- (3, 2);
    \draw[black!50!white, ultra thick] (-3, 1) -- (3, 1);
    \draw[black!50!white, ultra thick] (-3, -1) -- (3, -1);
    \draw[black!50!white, ultra thick] (-3, -3) -- (3, -3);

    \draw[black, dotted, ultra thick] (3, -3) -- (3, 3);
    \draw[black, dotted, ultra thick] (2, -3) -- (2, 3);
    \draw[black, dotted, ultra thick] (1, -3) -- (1, 3);
    \draw[black, dotted, ultra thick] (-1, -3) -- (-1, 3);
    \draw[black, dotted, ultra thick] (-3, -3) -- (-3, 3);
    \draw[black, dotted, ultra thick] (-3, 3) -- (3, 3);
    \draw[black, dotted, ultra thick] (-3, 2) -- (3, 2);
    \draw[black, dotted, ultra thick] (-3, 1) -- (3, 1);
    \draw[black, dotted, ultra thick] (-3, -1) -- (3, -1);
    \draw[black, dotted, ultra thick] (-3, -3) -- (3, -3);

    \filldraw[black] (3, 1) circle (4pt);
    \filldraw[black] (-1, 3) circle (4pt);
    \filldraw[black] (-3, -1) circle (4pt);
    \filldraw[black] (1, -3) circle (4pt);
    \filldraw[black] (2, 2) circle (4pt);
    
    

\end{tikzpicture}
        \begin{tikzpicture}[scale=0.5, outer sep=0]

    \fill[black!15!white] (3, -1) -- (3, 2) -- (2, 2) -- (2, 3) -- (-1, 3) -- (-1, 2) -- (-3, 2) -- (-3, -1) -- (-1, -1) -- (-1, -3) -- (2, -3) -- (2, -1) -- cycle;
    
    \draw[black!50!white, ultra thick] (3, -1) -- (3, 2);
    \draw[black!50!white, ultra thick] (2, -3) -- (2, 3);
    \draw[black!50!white, ultra thick] (1, -3) -- (1, 3);
    \draw[black!50!white, ultra thick] (-1, -3) -- (-1, 3);
    \draw[black!50!white, ultra thick] (-3, -1) -- (-3, 2);
    
    \draw[black!50!white, ultra thick] (-1, 3) -- (2, 3);
    \draw[black!50!white, ultra thick] (-3, 2) -- (3, 2);
    \draw[black!50!white, ultra thick] (-3, 1) -- (3, 1);
    \draw[black!50!white, ultra thick] (-3, -1) -- (3, -1);
    \draw[black!50!white, ultra thick] (-1, -3) -- (2, -3);

    \draw[black, dotted, ultra thick] (3, -1) -- (3, 2);
    \draw[black, dotted, ultra thick] (2, -3) -- (2, 3);
    \draw[black, dotted, ultra thick] (1, -3) -- (1, 3);
    \draw[black, dotted, ultra thick] (-1, -3) -- (-1, 3);
    \draw[black, dotted, ultra thick] (-3, -1) -- (-3, 2);
    
    \draw[black, dotted, ultra thick] (-1, 3) -- (2, 3);
    \draw[black, dotted, ultra thick] (-3, 2) -- (3, 2);
    \draw[black, dotted, ultra thick] (-3, 1) -- (3, 1);
    \draw[black, dotted, ultra thick] (-3, -1) -- (3, -1);
    \draw[black, dotted, ultra thick] (-1, -3) -- (2, -3);

    \filldraw[black] (3, 1) circle (4pt);
    \filldraw[black] (-1, 3) circle (4pt);
    \filldraw[black] (-3, -1) circle (4pt);
    \filldraw[black] (1, -3) circle (4pt);
    \filldraw[black] (2, 2) circle (4pt);

\end{tikzpicture}
        \begin{tikzpicture}[scale=0.5, outer sep=0]

    \fill[black!15!white] (2, -1) -- (2, 2) -- (1, 2) -- (1, 3) -- (-1, 3) -- (-1, 1) -- (-3, 1) -- (-3, -1) -- cycle;

    \draw[black!50!white, ultra thick] (2, -1) -- (2, 2);
    \draw[black!50!white, ultra thick] (1, -3) -- (1, 3);
    \draw[black!50!white, ultra thick] (-1, -1) -- (-1, 3);
    \draw[black!50!white, ultra thick] (-3, -1) -- (-3, 1);
    
    \draw[black!50!white, ultra thick] (-1, 3) -- (1, 3);
    \draw[black!50!white, ultra thick] (-1, 2) -- (2, 2);
    \draw[black!50!white, ultra thick] (-3, 1) -- (3, 1);
    \draw[black!50!white, ultra thick] (-3, -1) -- (2, -1);

    \draw[black, dotted, ultra thick] (2, -1) -- (2, 2);
    \draw[black, dotted, ultra thick] (1, -3) -- (1, 3);
    \draw[black, dotted, ultra thick] (-1, -1) -- (-1, 3);
    \draw[black, dotted, ultra thick] (-3, -1) -- (-3, 1);
    
    \draw[black, dotted, ultra thick] (-1, 3) -- (1, 3);
    \draw[black, dotted, ultra thick] (-1, 2) -- (2, 2);
    \draw[black, dotted, ultra thick] (-3, 1) -- (3, 1);
    \draw[black, dotted, ultra thick] (-3, -1) -- (2, -1);

    \filldraw[black] (3, 1) circle (4pt);
    \filldraw[black] (-1, 3) circle (4pt);
    \filldraw[black] (-3, -1) circle (4pt);
    \filldraw[black] (1, -3) circle (4pt);
    \filldraw[black] (2, 2) circle (4pt);

\end{tikzpicture}
        \begin{tikzpicture}[scale=0.5, outer sep=0]

    \fill[black!15!white] (2, 1) -- (2, 2) -- (-1, 2) -- (-1, -1) -- (1, -1) -- (1, 1) -- cycle;

    \draw[black!50!white, ultra thick] (2, 1) -- (2, 2);
    \draw[black!50!white, ultra thick] (1, -3) -- (1, 2);
    \draw[black!50!white, ultra thick] (-1, -1) -- (-1, 3);
    
    \draw[black!50!white, ultra thick] (-1, 2) -- (2, 2);
    \draw[black!50!white, ultra thick] (-1, 1) -- (3, 1);
    \draw[black!50!white, ultra thick] (-3, -1) -- (1, -1);

    \draw[black, dotted, ultra thick] (2, 1) -- (2, 2);
    \draw[black, dotted, ultra thick] (1, -3) -- (1, 2);
    \draw[black, dotted, ultra thick] (-1, -1) -- (-1, 3);
    
    \draw[black, dotted, ultra thick] (-1, 2) -- (2, 2);
    \draw[black, dotted, ultra thick] (-1, 1) -- (3, 1);
    \draw[black, dotted, ultra thick] (-3, -1) -- (1, -1);

    \filldraw[black] (3, 1) circle (4pt);
    \filldraw[black] (-1, 3) circle (4pt);
    \filldraw[black] (-3, -1) circle (4pt);
    \filldraw[black] (1, -3) circle (4pt);
    \filldraw[black] (2, 2) circle (4pt);

\end{tikzpicture}
        \caption{From left to right: the sets $B_0,\ldots,B_3$ from Algorithm \ref{alg:plane_hull} applied to the set of points $K$ (black dots) from Example \ref{ex:plane_grid}. Indeed, $\convrc K = B_3$.}
        \label{fig:grid}
    \end{figure}
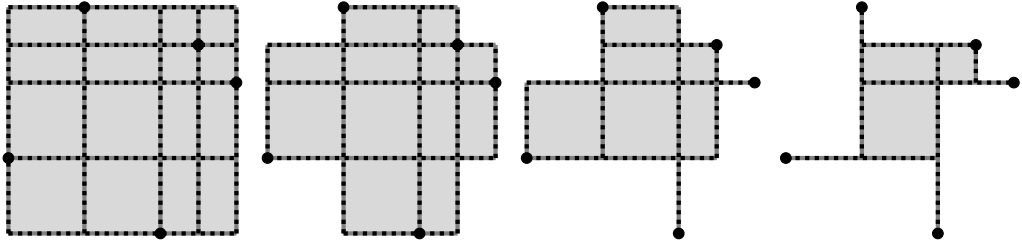
\end{example}

Given any set $K\subset \R^{n\times m}$, let
\begin{align*}
    K^{(0)} &= K, \\
    K^{(i+1)} &= \{ T_4 \hbox{ of points of } K^{(i)} \} \cup \{ \hbox{rank-one segments with extrema in } K^{(i)} \}, \\
    \conv_{T_4} K &= \cup_{i\geq 0} K^{(i)},
\end{align*}
and call the latter the \emph{$T_4$-convex hull} of $K$. This notion will be discussed further in Section \ref{sec:open_questions}. We state the following fact which is a straightforward consequence of other results in the literature, but is worth highlighting for later use. We provide the proof sketch for completeness.
\begin{corollary}\label{cor:convT4_plane}
    The rank-one convex hull of a finite set $K$ of $2\times 2$ diagonal matrices can be obtained as the $T_4$-convex hull of $K$ with only two iterations, namely 
    \[
    \convrc K = \conv_{T_4} K = K^{(0)}\cup K^{(1)} \cup K^{(2)}.
    \]
    In particular, if $K = \{K_1,\ldots,K_5\}$ with the first four points in $T_4$ configuration, then there exist $A,B,C \in \conv_{T_4} \{K_1,\ldots,K_4\}$ such that 
    \[
    \convrc K = \conv_{T_4} K = K^{(0)}\cup \conv_{T_4} \{K_1,\ldots,K_4\} \cup \conv_{T_4} \{K_5,A,B,C\}.
    \]
\end{corollary}
\begin{proof}
    Since rank-one convexity for $2\times 2$ diagonal matrices has Carath\'eodory number equal to five \cite[Proposition 5.3]{MatPle98:SeparateCH}, we can reduce the statement to configurations of five points $K=\{K_1,\ldots,K_5\}$ with a connected rank-one convex hull, since we can treat each connected component separately. It is easy to check that there are only two such combinatorially distinct configurations in the plane, namely configurations with a combinatorially different grid, see Figure \ref{fig:possible_5pts}, left and right. Without loss of generality, assume that the first four points are in a $T_4$ configuration. In both cases, there exist points $A,B,C \in \conv_{T_4} \{K_1,\ldots,K_4\} \subset K^{(1)}$ such that $\convrc K = K \cup K^{(1)} \cup \conv_{T_4} \{K_5,A,B,C\}$. These points $A, B, C$ are grid vertices arising from the grid lines originated at $K_1, K_4, K_5$, where $K_5$ lies in the orthant defined by $K_1$ and $K_4$ (see Figure \ref{fig:possible_5pts}). The claim follows.
    \begin{figure}[ht]
        \centering
        \begin{tikzpicture}[scale=0.7, outer sep=0]

    \fill[black!15!white] (2, 1) -- (2, 2) -- (-1, 2) -- (-1, -1) -- (1, -1) -- (1, 1) -- cycle;

    \draw[black!50!white, ultra thick] (2, 1) -- (2, 2);
    \draw[black!50!white, ultra thick] (1, -3) -- (1, 2);
    \draw[black!50!white, ultra thick] (-1, -1) -- (-1, 3);
    
    \draw[black!50!white, ultra thick] (-1, 2) -- (2, 2);
    \draw[black!50!white, ultra thick] (-1, 1) -- (3, 1);
    \draw[black!50!white, ultra thick] (-3, -1) -- (1, -1);

    \draw[black, dotted, ultra thick] (2, 1) -- (2, 2);
    \draw[black, dotted, ultra thick] (1, -3) -- (1, 2);
    \draw[black, dotted, ultra thick] (-1, -1) -- (-1, 3);
    
    \draw[black, dotted, ultra thick] (-1, 2) -- (2, 2);
    \draw[black, dotted, ultra thick] (-1, 1) -- (3, 1);
    \draw[black, dotted, ultra thick] (-3, -1) -- (1, -1);

    \filldraw[black] (3, 1) circle (4pt);
    \filldraw[black] (-1, 3) circle (4pt);
    \filldraw[black] (-3, -1) circle (4pt);
    \filldraw[black] (1, -3) circle (4pt);
    \filldraw[black] (2, 2) circle (4pt);

    \filldraw[black] (2,1) +(-4pt,-4pt) rectangle +(4pt,4pt) ;
    \filldraw[black] (-1,1) +(-4pt,-4pt) rectangle +(4pt,4pt) ;
    \filldraw[black] (-1,2) +(-4pt,-4pt) rectangle +(4pt,4pt) ;

    \node at (2,2.5) {$K_5$};
    \node at (-1.5,3) {$K_4$}; 
    \node at (-3,-1.5) {$K_3$}; 
    \node at (1.5,-3) {$K_2$}; 
    \node at (3,1.5) {$K_1$};
    
    \node at (-1.5,1) {$B$}; 
    \node at (-1.5,2) {$C$}; 
    \node at (2,0.6) {$A$}; 

\end{tikzpicture}
        \begin{tikzpicture}[scale=0.7, outer sep=0]

    \fill[black!15!white] (1, 1) -- (0, 1) -- (0, 2) -- (-1, 2) -- (-1, -1) -- (1, -1) -- (1, 1) -- cycle;

    \draw[black!50!white, ultra thick] (0, -1) -- (0, 2);
    \draw[black!50!white, ultra thick] (1, -3) -- (1, 1);
    \draw[black!50!white, ultra thick] (-1, -1) -- (-1, 3);
    
    \draw[black!50!white, ultra thick] (-1, 2) -- (0, 2);
    \draw[black!50!white, ultra thick] (-1, 1) -- (3, 1);
    \draw[black!50!white, ultra thick] (-3, -1) -- (1, -1);

    \draw[black, dotted, ultra thick] (0, -1) -- (0, 2);
    \draw[black, dotted, ultra thick] (1, -3) -- (1, 1);
    \draw[black, dotted, ultra thick] (-1, -1) -- (-1, 3);
    
    \draw[black, dotted, ultra thick] (-1, 2) -- (0, 2);
    \draw[black, dotted, ultra thick] (-1, 1) -- (3, 1);
    \draw[black, dotted, ultra thick] (-3, -1) -- (1, -1);

    \filldraw[black] (3, 1) circle (4pt);
    \filldraw[black] (-1, 3) circle (4pt);
    \filldraw[black] (-3, -1) circle (4pt);
    \filldraw[black] (1, -3) circle (4pt);
    \filldraw[black] (0, 2) circle (4pt);

    \filldraw[black] (0,1) +(-4pt,-4pt) rectangle +(4pt,4pt) ;
    \filldraw[black] (-1,1) +(-4pt,-4pt) rectangle +(4pt,4pt) ;
    \filldraw[black] (-1,2) +(-4pt,-4pt) rectangle +(4pt,4pt) ;

    \node at (0.5,2.5) {$K_5$};
    \node at (-1.5,3) {$K_4$}; 
    \node at (-3,-1.5) {$K_3$}; 
    \node at (1.5,-3) {$K_2$}; 
    \node at (3,1.5) {$K_1$};
    
    \node at (-1.5,1) {$B$}; 
    \node at (-1.5,2) {$C$}; 
    \node at (0.5,1.4) {$A$}; 
\end{tikzpicture}
        \caption{The two combinatorially possible configurations of five points in the plane with connected rank-one convex hull.}
        \label{fig:possible_5pts}
    \end{figure}
\end{proof}

\section{Triangular matrices}\label{sec:tri_matrices}
Identify $\R^3$ with the space of $2\times 2$ (upper) triangular matrices 
\[
A = \begin{pmatrix} x & z \\ 0 & y \end{pmatrix}
\]
and consider the associated cone of rank-one matrices $\Cm = \{xy=0\}\subset \R^3$. It is the union of two planes orthogonal to the $x$ and to the $y$ axis respectively. By abuse of notation denote $x(A)$, $y(A)$, $z(A)$ the first, second, and third coordinate of $A$, respectively. Let $\pi \colon \R^3 \to \R^2$ be the projection onto the first two coordinates. We may and will interpret $\R^2$ as the set of diagonal matrices. Then, $\pi(\Cm)=\{xy=0\}\subset \R^2$ is the cone $\mathcal{D}_{\rm diag}$ of rank-one directions for $2\times 2$ diagonal matrices: we say that the projection $\pi$ is \emph{compatible} with the two cones of triangular and diagonal rank-one matrices.
To study \emph{triangular hulls}, namely rank-one convex hulls in the case of triangular $2\times 2$ matrices, we will leverage the known construction of rank-one convex hulls in the plane via the compatibility of the rank-one cones. Note that in this setting $\pi(\convrc K) \subset \convrc \pi(K)$.
Since the inclusion-minimal nontrivial configuration in the plane is a $T_4$ \cite{Szekelyhidi05:rch2x2}, that is also the smallest nontrivial case in the triangular case, and we characterize it explicitly in the following proposition.  

\begin{proposition}\label{prop:T4}
    Let $K = \{K_1,\ldots,K_4\}\subset\R^3$ be a set of triangular $2\times 2$ matrices such that $\pi(K)$ is in a (possibly degenerate) $T_4$ configuration. Then, there exists a quadratic polynomial of the form $q(x,y,z) = z + \alpha x y + \beta x + \gamma y + \delta$ such that 
    \[
    \convrc K = ((\convrc \pi(K) )\times \R) \cap \{q = 0\}.
    \]
\end{proposition}
\begin{proof}
    Assume that $\pi(K)$ has no rank-one connection.
    Let us denote by $P_1,\ldots,P_4 \in \R^2$ the vertices of the square in the planar $T_4$ of $\pi(K)$. By the definition of $T_4$ configurations, there exist four scalars $\lambda_i \in (0,1)$ such that $P_i = \lambda_i \pi(K_{i-1}) + (1-\lambda_i) P_{i-1}$ for $i = 1,\ldots, 4$, where the indices are considered modulo $4$. Our goal is to construct points $Q_1,\ldots, Q_4$ such that $\pi(Q_i) = P_i$ and $Q_i = \lambda_i K_{i-1} + (1-\lambda_i) Q_{i-1}$, again modulo $4$. Therefore, the only unknowns of these equations are the heights, namely $z(Q_i)$. We obtain them by inverting the following linear system:
    \[
    \begin{pmatrix}
        z(K_1)\\
        z(K_2)\\
        z(K_3)\\
        z(K_4)
    \end{pmatrix}
    =
    \begin{pmatrix}
        \frac{\lambda_1 - 1}{\lambda_1} & \frac{1}{\lambda_1} & 0 & 0  \\
        0 & \frac{\lambda_2 - 1}{\lambda_2} & \frac{1}{\lambda_2} & 0 \\
        0 & 0 & \frac{\lambda_3 - 1}{\lambda_3} & \frac{1}{\lambda_3} \\
        \frac{1}{\lambda_4} & 0 & 0 & \frac{\lambda_4 - 1}{\lambda_4} \\
    \end{pmatrix}
    \begin{pmatrix}
        z(Q_1)\\
        z(Q_2)\\
        z(Q_3)\\
        z(Q_4)
    \end{pmatrix}.
    \]
    Let us now examine the quadratic surface $q=0$ through $K_1,\ldots,K_4,Q_1,\ldots,Q_4$. It necessarily contains four lines: one through the points $K_1,Q_2,Q_1$, another one the points $K_3,Q_4,Q_3$, another one the points $K_2,Q_3,Q_2$ and another one the points $K_4,Q_1,Q_4$. Two of these lines are parallel to the $x$ axis and two are parallel to the $y$ axis, which means that the equation of $q$ must satisfy 
    \[
    q \cap \{x = x_0\} = \hbox{ line in } y, z, \qquad q \cap \{y = y_0\} = \hbox{ line in } x, z.
    \]
    This forces the coefficients of $x^2, y^2, z^2, xz, yz$ to be zero. Hence, the quadric has equation (up to global constant)
    \[
    q(x,y,z) = z + \alpha xy + \beta x + \gamma y + \delta = 0,
    \]
    for some $\alpha, \beta, \gamma, \delta \in \R$.
    Such a quadratic polynomial is affine in rank-one directions, hence it is a rank-one convex function.
    Since $K\subset\{q = 0\}$, by equation \eqref{eq:rchull} applied to $f=q$ and $f=-q$, we get that $\convrc K \subset\{q = 0\}$. Moreover, all the points in $(\convrc \pi(K))  \times \R$ and in $\{q = 0\}$ belong to the triangular hull since we can split any measure following the rulings of $q$ to reach $K$. Hence, the claim follows.

    The case in which there are rank-one connections in $\pi(K)$ is analogous, despite the fact that some $P_i$ might coincide with points in $K$.
\end{proof}

See Figure \ref{fig:T4triang} for a plot of such a $T_4$ configuration and the associated quadric $\{q=0\}$. Notice that the proof of Proposition \ref{prop:T4} is constructive and gives an algorithm to compute the implicit equation of $q$. 
\begin{remark}\label{rmk:T4_all}
    Note that an argument analogous to the proof of Proposition \ref{prop:T4} shows that the rank-one convex hull of any four $2\times 2$ matrices in $\R^4 \sim \left(\begin{smallmatrix}
        x & z \\
        w & y
    \end{smallmatrix}\right)$ is semialgebraic and contained in a hypersurface of equation $xy-zw + \alpha x + \beta y + \gamma z + \delta w + \epsilon = 0$ intersected with the affine span of the four matrices. This partially uses \cite{Szekelyhidi05:rch2x2}.
\end{remark}

We will use $T_4$ configurations in $\R^3$ as our building blocks for more general triangular hulls. We start with a specific $5$-points configuration which highlights the subtleties of the $3$-dimensional case and the methods one can use to overcome them.

\begin{theorem}\label{thm:rch5pts}
    The triangular hull of the $5$-point set 
    \begin{equation}\label{eq:K5pts}
        K = \left\lbrace
        \begin{pmatrix}
            3 & 0 \\
            0 & 1
        \end{pmatrix},
        \begin{pmatrix}
            1 & 0 \\
            0 & -3
        \end{pmatrix},
        \begin{pmatrix}
            -3 & {-1} \\
            0 & -1
        \end{pmatrix},
        \begin{pmatrix}
            -1 & 0 \\
            0 & 3
        \end{pmatrix},
        \begin{pmatrix}
            2 & 2 \\
            0 & 2
        \end{pmatrix}
        \right\rbrace \subset \R^3
    \end{equation}
    is a semialgebraic set. It can be described using six linear and four quadratic polynomials, namely
    \begin{gather}
        \label{eq:5pts_linear}
        \ell_1 = x-2, \quad \ell_2 = x-1, \quad \ell_3 = x+1, \quad \ell_4 = y-2, \quad \ell_5 = y-1, \quad \ell_6 = y+1, \\
        \label{eq:quadrics_lower}
        {\color{MidnightBlue} q_1 = 60 z+5 x y-9 x-3 y+15}, \quad {\color{Apricot!80!black!80!red} q_2 = 45 z -29 x y+26 x-35 y+44}, \\
        \label{eq:quadrics_upper}
        {\color{ForestGreen} q_3 = 118 z-12 x y-57 x-19 y-36}, \quad {\color{RedViolet} q_4 = 118 z-143 x y+205 x+112 y-298}.
    \end{gather}
\end{theorem}
\begin{figure}[ht]
    \centering
    \begin{tikzpicture}[scale=0.85, outer sep=0]
    \fill[black!15!white] (2, 1) -- (2, 2) -- (-1, 2) -- (-1, -1) -- (1, -1) -- (1, 1) -- cycle;

    \draw[black!50!white, ultra thick] (1, 2) -- (1, -3);
    \draw[black!50!white, ultra thick] (-1, 1) -- (3, 1);
    \draw[black!50!white, ultra thick] (-1, -1) -- (-1, 3);
    \draw[black!50!white, ultra thick] (1, -1) -- (-3, -1);
    \draw[black!50!white, ultra thick] (-1, 2) -- (2, 2);
    \draw[black!50!white, ultra thick] (2, 1) -- (2, 2);

    \draw[black, dotted, ultra thick] (1, -1) -- (-3, -1);
    \draw[black, dotted, ultra thick] (-1, 1) -- (3, 1);
    \draw[black, dotted, ultra thick] (-1, 2) -- (2, 2);
    \draw[black, dotted, ultra thick] (-1, 3) -- (-1, -1);
    \draw[black, dotted, ultra thick] (1, 2) -- (1, -3);
    \draw[black, dotted, ultra thick] (2, 2) -- (2, 1);

    \filldraw[black] (3, 1) circle (3pt);
    \filldraw[black] (-1, 3) circle (3pt);
    \filldraw[black] (-3, -1) circle (3pt);
    \filldraw[black] (1, -3) circle (3pt);
    \filldraw[black] (2, 2) circle (3pt);

    \node at (2.7,0.6) {$\pi(K_1)$};
    \node at (0.2,-3) {$\pi(K_2)$}; 
    \node at (-2.7,-0.6) {$\pi(K_3)$}; 
    \node at (-0.2,3) {$\pi(K_4)$}; 
    \node at (2,2.4) {$\pi(K_5)$};
\end{tikzpicture}
    \input{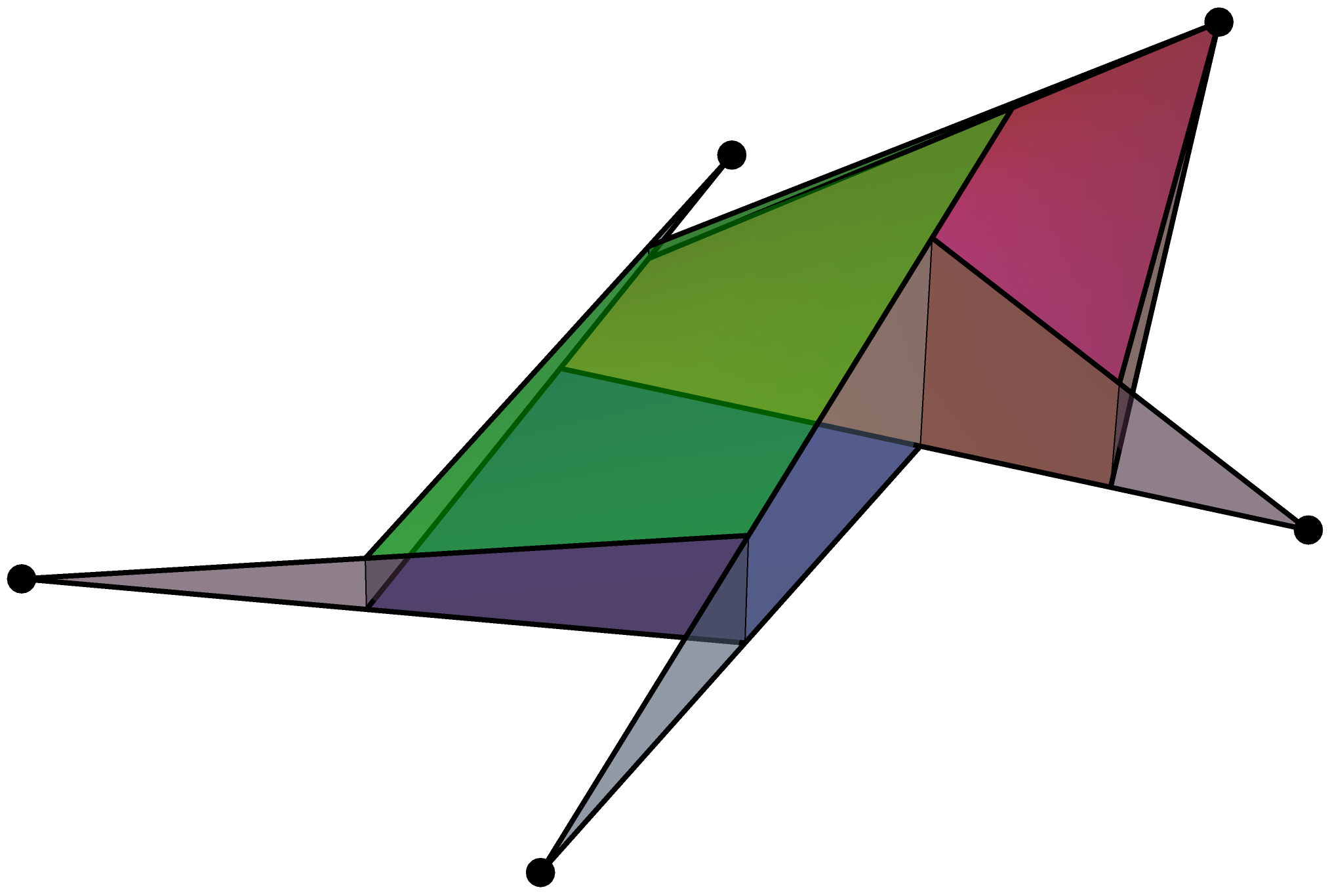}
    \caption{Semialgebraic sets from Theorem \ref{thm:rch5pts}. Left: the rank-one convex hull of $\pi(K)$. Right: the triangular hull of $K$.}
    \label{fig:thm5pts}
\end{figure}
\begin{proof}
    Let $S$ be the semialgebraic set in Figure~\ref{fig:thm5pts}. We will prove that $\convrc K = S$.
    The inclusion $S\subset \convrc K$ is easier. In fact, the black lines are in $\convrc K$ since we can keep splitting elementary measures along them and put all the weight on the points of $K$, evaluating at zero against the rank-one convex envelope of the square distance to $K$. As a consequence, also the points in the interior of the curvy squares defined by the quadrics are in the triangular hull, since we can split any measure to make it reach the boundary, and use the previous argument. Finally, also the points in the interior of $S$, by an analogous reasoning, belong to $\convrc K$.

    It is more subtle to prove that $S$ contains all the points of the triangular hull. Since $\pi(\convrc K) \subset \convrc \pi(K)$, then (a subset of) the linear forms in \eqref{eq:5pts_linear} appear in the boundary of $\convrc K$. We are left with the deletion of large $\pm z$ portions from $(\convrc \pi(K) ) \times \R$, and this is done using appropriate triangular-convex functions for Definition \ref{def:rch}. 
    Consider the blue quadric ${\color{MidnightBlue} q_1}$, obtained from the $T_4$ configuration $\{K_1,K_2,K_3,K_4\}$. Notice that ${\color{MidnightBlue} q_1}(K_i)=0$ for $i=1,\ldots,4$ and ${\color{MidnightBlue} q_1}(K_5)>0$. Therefore, $\convrc K \subset \{ -{\color{MidnightBlue} q_1}\leq 0 \}$. Analogously, one can show that $\convrc K \subset \{ {\color{ForestGreen} q_3}\leq 0 \}$. Unfortunately, it is not possible to repeat this reasoning for ${\color{Apricot!80!black!80!red} q_2}$ and ${\color{RedViolet} q_4}$, since both are zero at three $K_i$, positive at one, negative at another one. To solve this issue, define 
    \begin{align*}
    f_{\ell}(x,y,z) &= 
    \begin{cases}
        -{\color{Apricot!80!black!80!red} q_2}(x,y,z) & x\geq -1, y\geq 1;\\
        -\frac{3}{4}{\color{MidnightBlue} q_1}(x,y,z) & \hbox{otherwise}.
    \end{cases}
    \\
    f_{u}(x,y,z) &= 
    \begin{cases}
        {\color{RedViolet} q_4}(x,y,z) & x\geq 1, y\geq 1;\\
        {\color{ForestGreen} q_3}(x,y,z) & \hbox{otherwise}.
    \end{cases}
    \end{align*}
    These are rank-one convex functions by \cite[Lemma 3.1]{Sverak90:ExamplesrcFunct}, vanishing at all points $K_i$. Therefore, 
    \[
    \convrc K \subset \{ f_{\ell} \leq 0, f_{u}\leq 0\},
    \]
    and since $S = \big( (\convrc \pi(K) ) \times \R \big) \cap \{ f_{\ell} \leq 0, f_{u}\leq 0\}$, the claim follows. 
\end{proof}

The quadrics in \eqref{eq:quadrics_lower}, \eqref{eq:quadrics_upper} arise from (possibly degenerate, see Remark \ref{rmk:degenerateT4}) $T_4$ configurations, namely:
\begin{align*}
    & {\color{MidnightBlue} q_1} \longrightarrow  \{ K_1, K_2, K_3, K_4 \}, \\
    & {\color{Apricot!80!black!80!red} q_2} \longrightarrow\{ K_1, K_4, K_5, (-1,1,-\tfrac{4}{15}) \}, \\
    & {\color{ForestGreen} q_3} \longrightarrow\{ K_2, K_3, K_4, K_5 \}, \\
    & {\color{RedViolet} q_4} \longrightarrow \{ K_1, K_2, K_5, (-1,2,-\tfrac{155}{118}) \}.
\end{align*}
The two points that are not in $K$ but that play a fundamental role for the computation of the quadrics are obtained from the previous $T_4$. Specifically, $(-1,1,-\tfrac{4}{15})$ belongs to the blue $T_4$, whereas $(-1,2,-\tfrac{155}{118})$ belongs to the green one. The orange $T_4$ is displayed in Figure \ref{fig:T4triang}.
Analogously, it is possible to construct the rank-one convex hull of any $5$-points set of $2\times 2$ triangular matrices in the same way, via iterated computation of $T_4$ configurations.

Here we witness a big difference between standard convexity and rank-one convexity: 
the Carath\'eodory number of rank-one convexity is finite for $2\times 2$ diagonal matrices \cite[Proposition 5.3]{MatPle98:SeparateCH} and infinite in every other known case \cite{Kolar03:NonCompactHulls, Matousek01:DirectionalConvexity}.
These include our case of the rank-one convex hull of triangular $2\times 2$ matrices, which is examined in \cite[Example 3.5 $(ii)$]{MatPle98:SeparateCH}. Here the authors construct for any $n>0$ an explicit set of $n$ points $K\subset \R^3$ such that some of the points in $\convrc K$ do not belong to the hull of any $n-1$ points in $K$.

Our goal now is to generalize Theorem \ref{thm:rch5pts} to any set $K\subset\R^3$ of finitely many points. We will need the following result which relates the three-dimensional hull and the hull of its two-dimensional projection.
\begin{lemma}\label{lemma:equality_projection}
    Let $K = \{K_1,\ldots,K_N\}\subset\R^3$ be a set of triangular $2\times 2$ matrices. Then,
    \[
    \pi ( \convrc K) = \convrc \pi(K).
    \]
\end{lemma}
\begin{proof}
    For $N < 4$ the rank-one convex hull is either trivial or consists exactly of the rank-one connections of the points $K_i$ themselves.
    When $N=4$, if $\pi(K)$ is not in a (possibly degenerate) planar $T_4$ configuration, then $\convrc \pi(K)$ is a finite union of rank-one segments (or trivial), hence $\pi ( \convrc K) = \convrc \pi(K)$ is immediate. Otherwise, Proposition \ref{prop:T4} applies and the claim is proven. 
    
    Let now $N>4$. We already noticed that $\pi ( \convrc K) \subset \convrc \pi(K)$, so we have to prove the opposite inclusion. Let $(x,y)\in \convrc \pi(K)$. Since the Carath\'eodory number of rank-one convexity for diagonal $2\times 2$ matrices is $5$, without loss of generality, we assume that $(x,y)\in \convrc \{\pi(K_1),\ldots, \pi(K_5)\}$. By Corollary \ref{cor:convT4_plane}, the rank-one convex hull of these five diagonal $2\times 2$ matrices can be constructed by two iterations of $T_4$ constructions.
    If $(x,y) \in \convrc \{\pi(K_1),\ldots, \pi(K_4)\}$, then by case $N=4$ we have $(x,y)\in \pi (\convrc K)$. On the other hand, let $(x,y) \in \convrc \{A, B, C, \pi(K_5)\}$ for some points 
    \[
    A,B,C \in \convrc \{\pi(K_1),\ldots,\pi(K_4)\} = \pi(\convrc \{K_1,\ldots,K_4\})\subset \pi(\convrc
     K)
    \]
    that, together with $\pi(K_5)$ are in a (possibly degenerate) $T_4$ configuration, where for the equality we used the case $N=4$.
    Let $A',B',C' \in \convrc K$ be some points that project onto $A,B,C$, respectively.
    Then, applying again the case $N=4$, there exists $z \in \R$ such that $(x,y,z)\in \convrc \{A',B',C', K_5\} \subset \convrc K$, and the claim follows.
\end{proof}

In order to prove the next result, we introduce the following notation. We will call the \emph{upper hull} of $\convrc K$ the set
\[
\partial^u (\convrc K) = \Big\lbrace (x,y,z) \colon (x,y) \in \convrc \pi(K),\, z = \max \{c \colon (x,y,c) \in \convrc K\} \Big\rbrace
\]
of highest points above $\convrc \pi(K)$. Analogously, the \emph{lower hull} of $\convrc K$ is the set of lowest points above $\convrc \pi(K)$, namely
\[
\partial^\ell (\convrc K) = \Big\lbrace (x,y,z) \colon (x,y) \in \convrc \pi(K),\, z = \min \{c \colon (x,y,c) \in \convrc K\} \Big\rbrace.
\]
The existence of an upper and lower hull is guaranteed by Lemma \ref{lemma:equality_projection}.
Since in vertical planes rank-one convexity is just the usual convexity, it follows that $\convrc K$ is a vertical lamination between the sets $\partial^u (\convrc K)$, $\partial^\ell (\convrc K)$.

\begin{proof}[Proof of Theorem \ref{thm:semialg_triangular_hull}]
The proof will consist of an iterated reduction to simpler claims, that together establish the claim. Since $K$ is a finite set of points, $\convrc K$ is compact. Therefore, there exists $\Mmax\in \R_{>0}$ such that 
\[
\convrc K \subset \convrc \pi(K) \times (-\Mmax,\Mmax).
\]
We prove that the upper hull is covered by quadrics. The same proof works for the lower hull, by swapping maxima and minima. {We can assume without loss of generality that $\convrc K$ is connected, otherwise the argument applies to each connected component.}

We introduce the function
\[
u(x,y) \coloneqq \max \{ z \in \R \colon (x,y,z)\in \convrc K \},
\]
defined on $\convrc \pi(K)$. By Lemma~\ref{lemma:equality_projection}, $u$ is well-defined on $\convrc \pi(K)$.
We claim that $u$ is separately concave, namely, for every fixed $y$ (resp.\ $x$), the function $x \mapsto u(x,y)$ (resp.\ $y \mapsto u(x,y)$) is concave. Indeed, since $\convrc K$ is rank-one convex and vertical planes $\{y=\mathrm{const}\}$ and $\{x=\mathrm{const}\}$ are unions of rank-one lines, the sections
\[
\convrc K \cap \{y=\bar y\}, \qquad \convrc K \cap \{x=\bar x\}
\]
are convex subsets of $\R^2$, which implies the concavity of $u$ in each variable.

We now construct a function $w\colon \convrc \pi(K)\to\R$ by bilinear interpolation on each rectangular cell $R=[x_i,x_{i+1}]\times[y_j,y_{j+1}]$ of the grid. We define
\[
w(x,y) = (1-s)(1-t)u(x_i,y_j) + s(1-t)u(x_{i+1},y_{j}) + (1-s)tu(x_{i},y_{j+1}) + stu(x_{i+1},y_{j+1}),
\]
where
\[
s = \frac{x-x_i}{x_{i+1}-x_i}, \qquad t = \frac{y-y_j}{y_{j+1}-y_j}.
\]
We similarly define $w$  by linear interpolation on one-dimensional cells of the grid. 
By construction, $w$ is continuous and affine in each variable on every cell. It is separately concave, since along each horizontal and vertical grid line the function $w$ coincides with the piecewise affine interpolation of $u$, and the slopes of these interpolants are nonincreasing by concavity of $u$. Moreover, for every $K_\ell\in K$ we have
\begin{align}\label{eq:height_inequality}
z(K_\ell) \le u(\pi(K_\ell))  = w(\pi(K_\ell)).
\end{align}

{We aim to construct a separately concave extension of} the function $w$ to an open neighborhood of $\convrc\pi(K)$.
Let $\mathcal{G}$ be the grid associated to $\pi(K)$, and let $\mathcal{G}_\rc$ denote the subcomplex consisting of the closed faces of $\mathcal{G}$ contained in $\convrc\pi(K)$. By construction, $\mathcal{G}_\rc$ is obtained from $\mathcal{G}$ by iteratively removing
vertices which have at most one neighbor in each coordinate direction
(cf Algorithm~\ref{alg:plane_hull}).

We define a finite grid-neighborhood $\mathcal{G}_{\rm nh}$ of $\mathcal{G}_{\rc}$ by reversing the pruning procedure of Algorithm~\ref{alg:plane_hull}. 
We first extend the original grid so that the rank-one convex hull is strictly in its topological interior. Namely, if a cell of $\mathcal{G}_\rc$ intersects the boundary of the initial box $B=[x_{\min},x_{\max}]\times[y_{\min},y_{\max}]$, we extend the grid slightly outside $B$ by adding one extra row and one extra column of grid vertices, at distance $\eta>0$ from the corresponding sides of $B$ (see Figure 
\ref{fig:grid_neighborhood}, blue). Denote this by $\mathcal{G}_{\rm ext}$.
Now, starting from $\mathcal{G}_{\rc}$, we reinsert vertices of $\mathcal{G}_{\rm ext}$ one at a time, in the reverse order of their removal, until all vertices of $\mathcal{G}_{\rm ext}$ belonging to cells sharing at least one vertex with a cell of $\mathcal{G}_{\rc}$ have been restored: this will give $\mathcal{G}_{\rm nh}$.

Equivalently, we proceed in layers: first we add back the vertices which, at the current stage, have exactly two neighbors (these correspond to the square-marked points in Figure~\ref{fig:grid_neighborhood} at the first iteration); then we continue adding vertices which become eligible at the next stage (triangles in Figure~\ref{fig:grid_neighborhood}), and proceed like this (diamonds in Figure~\ref{fig:grid_neighborhood}).
If at some stage no vertex with two neighbors remains (e.g., after we add all diamonds), we continue by reinserting vertices with only one neighbor (circles in Figure~\ref{fig:grid_neighborhood}), until all vertices of $\mathcal{G}_{\rm nh}$ have been added.

\begin{figure}[h]
    \centering
    \begin{tikzpicture}[scale=0.5, outer sep=0]

    \begin{scope}[on background layer]
      \fill[YellowOrange, opacity=0.35]
        (2,1) -- (2,2) -- (-1,2) -- (-1,-1) -- (1,-1) -- (1,1) -- cycle;
    
      \draw[YellowOrange, opacity=0.35, line width=15pt,
            line cap=round, line join=round]
        (2,1) -- (2,2)
        (1,-3) -- (1,2)
        (-1,-1) -- (-1,3)
        (-1,2) -- (2,2)
        (-1,1) -- (3,1)
        (-3,-1) -- (1,-1);
    \end{scope}

    \fill[black!15!white] (2, 1) -- (2, 2) -- (-1, 2) -- (-1, -1) -- (1, -1) -- (1, 1) -- cycle;

    \draw[black!50!white, ultra thick] (2, 1) -- (2, 2);
    \draw[black!50!white, ultra thick] (1, -3) -- (1, 2);
    \draw[black!50!white, ultra thick] (-1, -1) -- (-1, 3);
    
    \draw[black!50!white, ultra thick] (-1, 2) -- (2, 2);
    \draw[black!50!white, ultra thick] (-1, 1) -- (3, 1);
    \draw[black!50!white, ultra thick] (-3, -1) -- (1, -1);

    \draw[MidnightBlue!50!white, ultra thick] (3.8, -1) -- (3.8, 2);
    \draw[MidnightBlue!50!white, ultra thick] (3.8, -1) -- (3, -1);
    \draw[MidnightBlue!50!white, ultra thick] (3.8, 1) -- (3, 1);
    \draw[MidnightBlue!50!white, ultra thick] (3.8, 2) -- (3, 2);
    \draw[MidnightBlue!50!white, ultra thick] (-3.8, -3) -- (-3.8, 1);
    \draw[MidnightBlue!50!white, ultra thick] (-3.8, -3) -- (-3, -3);
    \draw[MidnightBlue!50!white, ultra thick] (-3.8, 1) -- (-3, 1);
    \draw[MidnightBlue!50!white, ultra thick] (-3.8, -1) -- (-3, -1);
    \draw[MidnightBlue!50!white, ultra thick] (-3, 3.8) -- (1, 3.8);
    \draw[MidnightBlue!50!white, ultra thick] (-1, 3.8) -- (-1, 3);
    \draw[MidnightBlue!50!white, ultra thick] (1, 3.8) -- (1, 3);
    \draw[MidnightBlue!50!white, ultra thick] (-3, 3.8) -- (-3, 3);
    \draw[MidnightBlue!50!white, ultra thick] (2, -3.8) -- (-1, -3.8);
    \draw[MidnightBlue!50!white, ultra thick] (-1, -3.8) -- (-1, -3);
    \draw[MidnightBlue!50!white, ultra thick] (1, -3.8) -- (1, -3);
    \draw[MidnightBlue!50!white, ultra thick] (2, -3.8) -- (2, -3);    
       
    \draw[black!50!white, ultra thick] (3, -3) -- (3, 3);
    \draw[black!50!white, ultra thick] (2, -3) -- (2, 3);
    \draw[black!50!white, ultra thick] (1, -3) -- (1, 3);
    \draw[black!50!white, ultra thick] (-1, -3) -- (-1, 3);
    \draw[black!50!white, ultra thick] (-3, -3) -- (-3, 3);
    \draw[black!50!white, ultra thick] (-3, 3) -- (3, 3);
    \draw[black!50!white, ultra thick] (-3, 2) -- (3, 2);
    \draw[black!50!white, ultra thick] (-3, 1) -- (3, 1);
    \draw[black!50!white, ultra thick] (-3, -1) -- (3, -1);
    \draw[black!50!white, ultra thick] (-3, -3) -- (3, -3);

    \draw[black, dotted, ultra thick] (2, 1) -- (2, 2);
    \draw[black, dotted, ultra thick] (1, -3) -- (1, 2);
    \draw[black, dotted, ultra thick] (-1, -1) -- (-1, 3);
    
    \draw[black, dotted, ultra thick] (-1, 2) -- (2, 2);
    \draw[black, dotted, ultra thick] (-1, 1) -- (3, 1);
    \draw[black, dotted, ultra thick] (-3, -1) -- (1, -1);
    
    \filldraw[black] (3, 1) circle (4pt);
    \filldraw[black] (-1, 3) circle (4pt);
    \filldraw[black] (-3, -1) circle (4pt);
    \filldraw[black] (1, -3) circle (4pt);
    \filldraw[black] (2, 2) circle (4pt);

    \node at (1,3) {$\square$};
    \node at (3,2) {$\square$};
    \node at (-3,1) {$\square$};
    \node at (-1,-3) {$\square$};
    \node at (2,-1) {$\square$};
    
    \node at (2,3) {$\triangle$};
    \node at (3,-1) {$\triangle$};
    \node at (-3,2) {$\triangle$};
    \node at (-3,-3) {$\triangle$};
    \node at (2,-3) {$\triangle$};
    
    \node at (-3,3) {$\Diamond$};
    \node at (3,3) {$\Diamond$};
    \node at (3,-3) {$\Diamond$};

    \node[scale=0.6,thick] at (3.8,1) {$\bigcirc$};
    \node[scale=0.6,thick] at (-3.8,-1) {$\bigcirc$};
    \node[scale=0.6,thick] at (-1,3.8) {$\bigcirc$};
    \node[scale=0.6,thick] at (1,-3.8) {$\bigcirc$};


    \foreach \p in {
      (3.8,-1),(3.8,2),
      (-3.8,-3),(-3.8,1),
      (-3,3.8),(1,3.8),
      (2,-3.8),(-1,-3.8)
    }
    {
      \node[draw=black, fill=none,
      star, star points=5, minimum size=8pt, star point ratio=1.8, inner sep=0pt] at \p {};
    }
\end{tikzpicture}
    \caption{The extension procedure of $w$ onto the orange neighborhood $O$ of $\convrc \pi(K)$. The vertices are reinserted in layers: squares, then triangles, then diamonds, then open circles, then stars. In blue the extension of the initial box.}
    \label{fig:grid_neighborhood}
\end{figure}

We extend the vertex values  inductively to the vertices of $\mathcal{G}_{\rm nh}$. Suppose that the values have been assigned on the current subgrid, and let $P$ be the next vertex to be reinserted. By construction, $P$ has at most one neighbor in each coordinate direction among the vertices already present. 
Consequently, along each horizontal and vertical grid line, the discrete concavity inequalities involving $P$ arise only on one side of $P$, since the vertices on the other side have not yet been reinserted.

We assign to $P$ a value $\widetilde w(P)$ such that all these inequalities are satisfied. Since only finitely many such inequalities arise at this step, and each of them becomes valid for $\widetilde w(P)$ sufficiently small, such a choice is always possible. Moreover, by the order in which the vertices are reinserted, previously established inequalities do not involve $P$, hence they remain valid.
After finitely many steps, this yields an extension of $w$ to all vertices of $\mathcal{G}_{\rm nh}$ such that the discrete concavity inequalities hold along every horizontal and vertical grid line.

Let $G_{\rm nh}$ denote the union of all closed cells of $\mathcal{G}_{\rm nh}$ in $\R^2$. By construction of the extended grid, $\convrc\pi(K)\subset \operatorname{int}(G_{\rm nh})$. Choose $\rho>0$ small enough such that
\[
O \coloneqq \{p\in\R^2 \colon \dist(p,\convrc\pi(K))<\rho\} \subset \operatorname{int}(G_{\rm nh}).
\]
Then $O$ is an open neighborhood of $\convrc\pi(K)$.
Let $\widetilde w$ be the piecewise bilinear interpolation on this enlarged grid. By construction, $\widetilde w$ is separately concave on $O$ and coincides with $w$ on $\convrc\pi(K)$.
It follows that the function
\[
\widetilde{f}(x,y,z)\coloneqq z-\widetilde w(x,y)
\]
is rank-one convex on the open set $\Omega = O \times(-\Mmax,\Mmax)$. By \cite[Lemma 4.10]{Kirchheim03:RigidityGeometry}, which generalizes \cite[Lemma 3.5]{MulSve99:ConvexIntegration} {and} \cite[Lemma 2.3]{MulSve03:CounterexRegularity}, there exists a rank-one convex function $f$ defined on $\R^3$ that coincides with $\widetilde{f}$ on $\convrc K$. Hence, this function satisfies $f=z-w(x,y)$ on $\convrc K$ and $f\le 0$ on $K$ {by \eqref{eq:height_inequality}}. By \eqref{eq:rchull}, it follows that $f\le 0$ on $\convrc K$, hence
\[
u(x,y) \le w(x,y) \qquad \text{for all } (x,y)\in \convrc \pi(K).
\]
On the other hand, since $u$ is separately concave and $w$ is obtained by affine interpolation of the values of $u$ at the grid vertices, then
\[
u(x,y) \ge w(x,y)
\qquad \text{for all } (x,y)\in \convrc \pi(K).
\]
{To see this, first notice that if $(x,y)$ is part of the grid, then $u(x,y)=w(x,y)$ by definition of $w$. If $(x,y)$ lies on a grid segment in $\convrc K$, say $(x,y)=\lambda P_1+(1-\lambda)P_2$ where $P_1,\,P_2$ are grid points defining a segment parallel to one of the axes, then by separate concavity of $u$
\begin{align}\label{eq:sep_concavity}
u(x,y)\geq \lambda u(P_1)+(1-\lambda)u(P_2)=\lambda w(P_1)+(1-\lambda)w(P_2)=w(x,y).
\end{align}
Finally, if $(x,y)$ lies in the interior of one of the rectangles in the grid, we can write it as a convex combination of two points $P_3,\,P_4$ lying on the boundary of the rectangle and defining a segment parallel to the axes. We then have
\[
u(x,y)\geq \mu u(P_3)+(1-\mu)u(P_4)\geq \mu w(P_3)+(1-\mu)w(P_4)=w(x,y),
\]
where the second inequality is \eqref{eq:sep_concavity}.}

Hence $u=w$ on $\convrc \pi(K)$.
Thus the upper hull is determined by its values over the vertices of the planar grid. Once the points
\(
(P,u(P))
\)
are known at vertices $P$ of the grid of $\convrc\pi(K)$, the whole upper hull is obtained by rank-one lamination over each cell of the grid.

Analogously, we can describe the lower hull with a piecewise bilinear, separately convex function $\ell(x,y)$. Then
\[
\convrc K = \{ (x,y,z) \in \R^3 \colon (x,y)\in\convrc\pi(K),\; z-u(x,y)\leq 0,\; z-\ell(x,y)\geq 0 \},
\]
so in particular, it is semialgebraic.
In other words, we can write the triangular hull of $K$ as the union of finitely many semialgebraic sets, all of them defined by either quadratic (supported on grid cells) or linear polynomial equations and inequalities. The claim follows.
\end{proof}

We record a few consequences of the proof above. First, it is not true in
general that a rank-one convex function defined on a rank-one convex set
admits a rank-one convex extension to an arbitrary neighborhood of that set.
The extension procedure used here relies crucially on the fact that the
planar rank-one convex hull is obtained by a finite grid-pruning algorithm.
This combinatorial structure allows for a one-sided inductive construction
of separately concave extensions.
Moreover, the proof shows that once the highest and lowest points of
$\convrc K$ above each planar grid vertex are known, the whole hull is
obtained by standard rank-one laminations over the grid cells. In particular,
the rank-one extreme points, in the sense of
\cite[Definition~3.8]{Kirchheim03:RigidityGeometry}, are among these lifted
grid vertices.

As a consequence, we obtain a localization result for certain lower-dimensional subsets, showing that on suitable vertical strips associated with grid edges, triangular convexity reduces to classical convexity.
\begin{corollary}
    Let $K\subset\R^3$ be a finite set of triangular matrices. Fix $\Mmax$ such that $\convrc K \subset \R^2\times (-\Mmax,\Mmax)$. Let $P_1, P_2 \in \R^2$ belong to the grid associated to $\pi(K)$, such that the rank-one segment $[P_1,P_2]$ does not contain grid vertices in its interior. Then, for $\blacksquare = [P_1,P_2]\times (-\Mmax,\Mmax)$ we have
    \[
    \convrc K \cap \blacksquare = \conv ( \convrc K \cap\, \partial_{\rm rel} \blacksquare ).
    \]
\end{corollary}

The corollary is not an immediate consequence of the usual localization
theorem. Indeed, \cite[Theorem~4.7]{Kirchheim03:RigidityGeometry} localizes
the computation of rank-one convex, more generally directional convex, hulls
to full-dimensional sets. The full-dimensionality assumption is essential:
for example, intersecting the $T_4$ configuration from Example~\ref{ex:T4}
with a sufficiently long segment through the origin gives a counterexample
to such a statement for arbitrary lower-dimensional restrictions. The
corollary above is therefore a special lower-dimensional localization result,
valid for the vertical strips \(\blacksquare\) arising from grid edges.

\color{black}

\section{Directional convexity in three dimensions}\label{sec:directional_convexity}

A natural generalization of rank-one convexity is \emph{directional convexity} which substitutes the cone of rank-one matrices with any set of directions \cite{Matousek01:DirectionalConvexity}. We take this a step further and consider convexity with respect to an arbitrary cone $\mathcal{D}$, meaning that $\lambda A \in \mathcal{D}$ whenever $A \in \mathcal{D}$ and $\lambda \geq 0$.
\begin{definition}
    A function $f:\R^{n\times m} \to \R$ is said to be \emph{$\mathcal{D}$-convex} with respect to $\mathcal{D}\subset \R^{n\times m}$ if the restriction $t\mapsto f(A+tB)$ is a convex function for every $A\in\R^{n\times m}$ and for every $B\in\mathcal{D}$. 
    The \emph{$\mathcal{D}$-convex hull} of a compact set $K\subset\R^{n\times m}$ is the set
\begin{equation*}
    \convD K = \{A\in \R^{n\times m} \colon f(A) \leq \max f(K)\, \text{for all $\mathcal{D}$-convex } f \}.
\end{equation*}
\end{definition}
We can translate all the language that we developed in Section \ref{sec:background} to this setting.
Consider for instance a cone $\mathcal{D}\subset\R^2$ of finitely many directions in the plane. If there are only two directions, then we reduce to rank-one convexity of $2\times 2$ matrices. However, also in the case of more directions some of the results for triangular matrices can be adapted. This is true for Algorithm \ref{alg:plane_hull}. The procedure consists again in constructing the $\mathcal{D}$-grid associated to the finite set $K\subset\R^2$, which is a line arrangement of the cone $\mathcal{D}$ translated at every point of $K$. Then, one has to remove the redundant vertices of the $\mathcal{D}$-grid, and this step becomes a bit more subtle but still feasible. The analog of Theorem \ref{thm:plane} holds true (see \cite[Theorem 1.1]{FraMat09:DHullsPlane}) and therefore the $\mathcal{D}$-convex hull of finitely many points in $\R^2$ with respect to finitely many directions is semialgebraic. Figure \ref{fig:direct_hull_5pts} displays the $\mathcal{D}$-convex hull of the set $K$ from Example \ref{ex:plane_grid}, with respect to the cone $\mathcal{D} = \{xy(y-\frac{2}{3}x)=0\}\subset\R^2$.
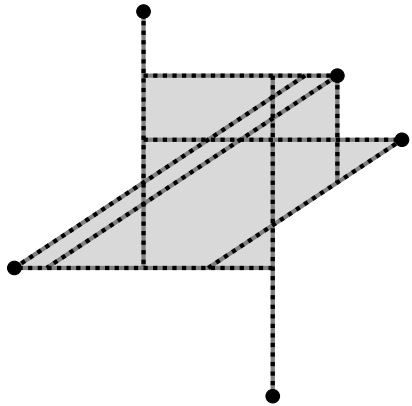
\begin{figure}[ht]
    \centering
    \begin{tikzpicture}[scale=0.85, outer sep=0]
    \fill[black!15!white] (1, -1) -- (1, -1/3) -- (3, 1) -- (2, 1) -- (2, 2) -- (-1, 2) -- (-1, 1/3) -- (-3, -1) -- cycle;

    \draw[black!50!white, ultra thick] (1, 2) -- (1, -3);
    \draw[black!50!white, ultra thick] (-1, 1) -- (3, 1);
    \draw[black!50!white, ultra thick] (-1, -1) -- (-1, 3);
    \draw[black!50!white, ultra thick] (1, -1) -- (-3, -1);
    \draw[black!50!white, ultra thick] (-1, 2) -- (2, 2);
    \draw[black!50!white, ultra thick] (2, 1/3) -- (2, 2);

    \draw[black!50!white, ultra thick] (0, -1) -- (3, 1);
    \draw[black!50!white, ultra thick] (1.5, 2) -- (-3, -1);
    \draw[black!50!white, ultra thick] (2, 2) -- (-2.5, -1);

    \draw[black, dotted, ultra thick] (1, -1) -- (-3, -1);
    \draw[black, dotted, ultra thick] (-1, 1) -- (3, 1);
    \draw[black, dotted, ultra thick] (-1, 2) -- (2, 2);
    \draw[black, dotted, ultra thick] (-1, 3) -- (-1, -1);
    \draw[black, dotted, ultra thick] (1, 2) -- (1, -3);
    \draw[black, dotted, ultra thick] (2, 2) -- (2, 1/3);
    \draw[black, dotted, ultra thick] (2, 2) -- (-2.5, -1);
    \draw[black, dotted, ultra thick] (1.5, 2) -- (-3, -1);
    \draw[black, dotted, ultra thick] (0, -1) -- (3, 1);

    \filldraw[black] (3, 1) circle (3pt);
    \filldraw[black] (-1, 3) circle (3pt);
    \filldraw[black] (-3, -1) circle (3pt);
    \filldraw[black] (1, -3) circle (3pt);
    \filldraw[black] (2, 2) circle (3pt);

\end{tikzpicture}
    \caption{The $\mathcal{D}$-convex hull of five points in the plane, taken from Example \ref{ex:plane_grid}, with respect to the cone $\mathcal{D} = \{xy(y-\frac{2}{3}x)=0\}\subset\R^2$.}
    \label{fig:direct_hull_5pts}
\end{figure}
After these considerations and the results presented in Section \ref{sec:tri_matrices}, a natural question follows.
\begin{question}\label{q:more_directions}
    Let $\mathcal{D} = \mathcal{D}_2 \times\R \subset \R^3$ be a cone, with $\mathcal{D}_2\subset\R^2$ a cone of finitely many directions. Given a finite set of points $K\subset\R^3$, is the $\mathcal{D}$-convex hull of $K$ semialgebraic?
\end{question}
Theorem \ref{thm:semialg_triangular_hull} gives an affirmative answer to this question in the case $\mathcal{D}_2 = \{xy=0\}$. In the following, we make some considerations to support our belief that Question \ref{q:more_directions} is \emph{not} true in general. For our argumentation, we will focus on the case with three directions
\begin{equation}\label{eq:3directions_plane}
\mathcal{D}_2 = \{(a_1 x + b_1 y) (a_2 x + b_2 y) (a_3 x + b_3 y) = 0\}\subset\R^2, \quad a_i, b_i \in\R,\; i = 1,2,3 .
\end{equation}
We will consider again $\pi:\R^3\to \R^2$ to be the projection map onto the first two coordinates.

With three directions, the inclusion-minimal configuration (when the initial points are not connected to each other by lines in $\mathcal{D}$) is a $T_3$. However, a $T_3$ configurations in $3$-dimensional space lies necessarily in the plane spanned by the three points: the linear form of the plane is $\mathcal{D}$-concave and convex, so it allows us to remove both open half spaces from the directional convex hull. This consideration proves the following fact.
\begin{lemma}\label{lem:direc_T3}
    Let $K=\{K_1,K_2,K_3\}\subset\R^3$ be a set of points such that $\pi(K)$ is in a (possibly degenerate) $T_3$ configuration with respect to $\mathcal{D}_2$ as in \eqref{eq:3directions_plane}. Let $H$ be the plane spanned by the points $K_1,K_2,K_3$. Then 
    \[
    \convD K = ((\conv_{\mathcal{D}_2} \pi(K)) \times\R) \cap H.
    \]
\end{lemma}
Based on the results for triangular matrices, since also in this directional convexity context the planar grid algorithm works, one would hope to use that with now triangles as building blocks instead of rectangles. If the $\mathcal{D}$-grid associated to $K$ would subdivide $\conv_{\mathcal{D}_2} \pi(K)$ into triangles, with edge directions contained in $\mathcal{D}_2$, then it would be possible to repeat the arguments in Theorem \ref{thm:semialg_triangular_hull} with triangles and planes instead of rectangles and quadrics, to conclude that the $\mathcal{D}$-convex hull in this case is semialgebraic. 
However, as Figure \ref{fig:direct_hull_5pts} shows, there are many open two-dimensional cells of the grid that are not triangles. This is a main difference between having only two or more directions in the plane: with two directions, at every grid point the cone $\mathcal{D}$ translated at that point is already part of the grid; with more directions, this is not true generically. 

In order to avoid this issue, one could try and refine the grid associated to $K$ an construct a \emph{grid compatible with $K$ and $\mathcal{D}_2$} via an inductive procedure: at every step, add to the line arrangement all the translations of $\mathcal{D}_2$ to the newly found grid points (only those contained in $\conv_{\mathcal{D}_2}K$). If such a procedure terminates in finitely many steps, we call the grid itself finite. In that case, one could attempt to adapt the techniques from Theorem \ref{thm:semialg_triangular_hull}: 
\begin{enumerate}
    \item[(i)] each vertex of the grid has a highest and lowest point in $\R^3$ which project onto it;
    \item[(ii)]\label{step:triangles} for each grid segment and triangle in the plane, construct the respective segment or triangle in the upper/lower hull using Lemma \ref{lem:direc_T3};
    \item[(iii)] prove that, up to vertical lamination, the union of segments and triangles from step (ii) (plus additional isolated vertical lines) is the $\mathcal{D}$-convex hull of $K$.
\end{enumerate}
This raises the natural question: does any tuple $(K, \mathcal{D}_2)$ admit a compatible finite grid? Or more in general, can the $\mathcal{D}_2$-convex hull of $K$ be triangulated using only triangles with edge directions belonging to $\mathcal{D}_2$? Unfortunately, the answer to this question is negative, as the following example shows. 
\begin{example}\label{ex:2d_infinite}
    Let $K = \{ K_1, K_2, K_3, K_4\} \subset\R^2$ be the vertices of a parallelogram $P$ with acute angle $\frac{\pi}{3}$ and edge lengths $a, b$ with $\frac{a}{b}\not\in\mathbb{Q}$. Consider the cone 
    \begin{equation}\label{eq:D2_equilateral}
    \mathcal{D}_2 = \{ y (y - \tan \tfrac{\pi}{3} x) (y + \tan \tfrac{\pi}{3} x) =0\}.
    \end{equation}
    Assume by contradiction that the $\mathcal{D}_2$-convex hull of $K$, namely $P$, admits a triangulation with finitely many triangles with edge directions in $\mathcal{D}_2$, and therefore equilateral. By \cite[Corollary 7.6]{KLLLT23:TilingOfRectangles}, this would imply that there exists a scaling $\lambda <\infty$ such that all the triangles in the scaled triangulation of $\lambda P$ have integer edge length. Therefore, $\lambda a, \lambda b \in \mathbb{Z}$, which contradicts the hypothesis $\frac{a}{b}\not\in\mathbb{Q}$. 
\end{example}

\begin{remark}
    As pointed out in Section \ref{subsec:calculus_variations}, the cone of rank-one matrices is the wave cone of the curl operator. However, other differential operators have other wave cones, providing a motivation for directional convexity, and for Example \ref{ex:2d_infinite}. For instance, the cone $\mathcal D_2$ in \eqref{eq:D2_equilateral} is the wave cone of the operator $\mathcal A$ defined on $v\colon \R^3\to \R^{3}$, given by
    $$
    \mathcal Av=(\partial_2 v_1,(\partial_2-\tan\tfrac{\pi}{3}\partial_1) v_2,(\partial_2+\tan\tfrac{\pi}{3}\partial_1) v_3).
    $$
\end{remark}

The infinite nature of the planar compatible grid in Example \ref{ex:2d_infinite} makes the approach of Theorem \ref{thm:semialg_triangular_hull} fail, as that would lead to a $\mathcal{D}$-convex hull with infinitely many planes intersecting its boundary. The missing detail to disprove a version of Conjecture \ref{conj:main} for directional convexity is that some of these planes could a priori coincide (also in the rank-one convexity case this happens for some quadrics). Therefore, one would have to prove that, possibly after small perturbation of $K$, only finitely many such planes can coincide, and therefore $\convD K$ cannot be semialgebraic. We leave this for further investigation, and state it as an open question.
\begin{question}
    Let $K\subset\R^3$ project onto the points in Example \ref{ex:2d_infinite}, and let $\mathcal{D} = \mathcal{D}_2 \times \R$ with $\mathcal{D}_2$ as in \eqref{eq:D2_equilateral}. Is it true that, up to small perturbations of $K$ in the vertical direction, $\convD K$ is not semialgebraic, and therefore contradicts Conjecture \ref{conj:main}?
\end{question}

\section{Open questions}\label{sec:open_questions}
We collect in this final section more questions related to Conjecture \ref{conj:main} for rank-one convexity and to the methods we used to prove it in the triangular case at hand.

Since the Carath\'eodory number of triangular convexity is infinite, we consider then a different way to construct the rank-one convex hull via a \emph{finite} procedure.
The following question is inspired by the $5$-point configuration from Theorem \ref{thm:rch5pts}, as well as by \cite{AngGar024:2+1hull}. Consider the set obtained from $K$ by consecutively computing $T_4$ configurations and laminations. Recall that the \emph{$T_4$-convex hull} of $K$ is
\begin{align*}
    K^{(0)} &= K, \\
    K^{(i+1)} &= \{ T_4 \hbox{ of points of } K^{(i)} \} \cup \{ \hbox{rank-one segments with extrema in } K^{(i)} \}, \\
    \conv_{T_4} K &= \cup_{i\geq 0} K^{(i)}.
\end{align*}
In the plane, by Corollary \ref{cor:convT4_plane}, we have that $\convrc K = \conv_{T_4} K$, and the same holds for $K$ as in Theorem \ref{thm:rch5pts}. 
\begin{question}\label{q:T4hull}
    ~
    \begin{enumerate}
        \item Let $K\subset \R^3$ be a finite set of $2\times 2$ triangular matrices. Does $\convrc K = \conv_{T_4} K$ always hold? 
        \item Is $\conv_{T_4} K$ a finite procedure? Namely, for a finite set $K$, is $\conv_{T_4} K = \cup_{i=0}^N K^{(i)}$ for some $N>0$, with $K{(i)}$ defined as above?
    \end{enumerate}
\end{question}
By \cite[Theorem 1]{Szekelyhidi05:rch2x2}, $\conv_{T_4} K$ is nontrivial if and only if $\convrc K$ is nontrivial, therefore the $T_4$ convex hull is a nontrivial inner approximation of the triangular hull. A positive answer to Question \ref{q:T4hull} would provide an algorithm for the computation of triangular hulls of finitely many points.
\begin{remark}
    Question \ref{q:T4hull} holds true for diagonal $2\times 2$ matrices (Corollary \ref{cor:convT4_plane}), but it cannot be generalized much further. For $2\times 2$ symmetric matrices, for which \cite[Theorem 1]{Szekelyhidi05:rch2x2} still holds, the $5$-points configuration in \cite{Pompe10:5gradProblem} provides a negative answer to Question \ref{q:T4hull}.
\end{remark}

A negative answer to the next question would provide a different proof to Theorem \ref{thm:semialg_triangular_hull} with the explicit construction of the upper/lower rank-one functions that cut out the triangular hull vertically. We state the question for the upper hull, but the same can be asked for the lower hull.
\begin{question}\label{q:sign_quadric}
    Let $K\subset \R^3$ be a finite set of $2\times 2$ triangular matrices. Consider the associated planar grid, a two-dimensional cell $Q$ of the grid, and let $Q_1, Q_2, Q_3, Q_4$ be the points in $\convrc K$ with largest $z$ coordinate that project onto the vertices of $Q$. Let $q$ be the quadric that describes the (degenerate) $T_4$ configuration of $\{Q_1, Q_2, Q_3, Q_4\}$. Can it happen that in each of the four quadrants identified by $Q$, the quadric $q$ has strictly positive sign on at least one point of $K$ in the quadrant? 
\end{question}
If the answer to Question \ref{q:sign_quadric} were negative, then one could construct the upper hull of the triangular hull in a similar way as for $f_u$ in the proof of Theorem \ref{thm:rch5pts}, by gluing appropriate quadrics.
This property could also be used as stopping criterion for a tentative algorithm, provided Question \ref{q:T4hull} were true, or as a tool to prove Question \ref{q:T4hull}.

We have seen the interpretation of the rank-one convex hull in terms of laminates of finite order. We introduce the notion of \emph{consecutive} laminates. A sequence $\{\nu_i\}_i$ of laminates of finite order is called consecutive if $\nu_{i+1}$ is obtained from $\nu_i$ via one elementary splitting. 
\begin{question}
    Let $K\subset\R^{n\times m}$ be a finite set. Is
    \[
     \convrc K = \{A\in \R^{n\times m}\colon \inf \{\langle \nu , d_K^2 \rangle\colon\nu \hbox{ consecutive laminates, } \overline{\nu} = A\} = 0\} ?
    \]
\end{question}
This restriction would allow some useful geometric reasoning when thinking in terms of consecutive splittings in rank-one directions. The question itself can also be specialized to $n=m=2$, or generalized to any compact set $K$.

\printbibliography
~\\

\small
\noindent \textsc{Chiara Meroni} \\
\textsc{ETH Institute for Theoretical Studies, Z\"urich, Switzerland} \\
\url{chiara.meroni@eth-its.ethz.ch} \\

\noindent \textsc{Bogdan~Rai\cb{t}\u{a}} \\
\textsc{Department of Mathematics and Statistics, Georgetown University,\\ Washington, DC, United States of America} \\
\url{br607@georgetown.edu} \\

\end{document}